\documentclass[article]{interact}

\usepackage[natbibapa,nodoi]{apacite}
\usepackage[normalem]{ulem}
\usepackage[latin1]{inputenc}
\usepackage{amsmath,amssymb,amsthm,amsbsy,amsfonts}
\usepackage{enumerate}
\usepackage{mathabx}
\newtheorem{theorem}{Theorem}
\usepackage{makeidx}
\usepackage{graphicx,tikz,pgfplots,subfigure}
\usetikzlibrary{positioning, backgrounds, calc, arrows, shapes.multipart}
\usepackage{comment}
\usepackage{algorithm}
\usepackage{algpseudocode}
\newtheorem{lemma}{Lemma}

\newtheorem{remark}{Remark}
\newtheorem{corollary}{Corollary}
\newtheorem{defi}{Definition}
\usepackage{color}

\newcommand{\C}{\mathcal{C}}
\newcommand{\G}{\mathcal{G}}

\pgfplotsset{compat=1.17}

\begin{document}

\author{
\name{Francesco Lo Iudice \textsuperscript{a}\thanks{For any correspondance on this work contact Francesco Lo Iudice, Email: francesco.loiudice2@unina.it.}, Anna Di Meglio\textsuperscript{a}, Fabio Della Rossa\textsuperscript{b}, and Francesco Sorrentino\textsuperscript{c}}
\affil{\textsuperscript{a}Department of Information Technology and Electrical Engineering, University of Naples Federico II, 80125 Naples, Italy; \textsuperscript{b}Department of Electronics, Information, and Bioengineering, Politecnico of Milan, Italy;
\textsuperscript{c}Department of Electronics, Information, and Bioengineering, Politecnico of Milan, Italy; Department of Mechanical Engineering, MSC01 1150 1 University of New Mexico, Albuquerque, NM 87131, USA }
}

\title{Controlling consensus in networks with symmetries}

\maketitle

\begin{abstract}
We study networks with linear dynamics where the presence of symmetries of the pair $(A,B)$, induces a partition of the network nodes in clusters and the matrix $A$ is not restricted to be in Laplacian form. For these networks, an invariant \emph{group consensus subspace} can be defined, in which the nodes in the same cluster evolve along the same trajectory in time. We prove that the network dynamics is uncontrollable in directions orthogonal to this subspace. Under the assumption that the dynamics parallel to this subspace is controllable, we design optimal controllers that drive the group consensus dynamics towards a desired state. Then, we consider the problem of selecting additional control inputs that stabilize the group consensus subspace and obtain bounds on the minimum number of additional inputs and driver nodes needed to this end. Altogether, our results indicate that it is possible to design \emph{independently} the control action along and transverse to the group consensus subspace.
\end{abstract}

\section{Introduction}

The number of real-word systems modeled as complex networks is ever increasing, and ranges from natural \citep{VaSo:96,SeEa:09}, technological, \citep{StTj:16,YuZh:12} and social systems \citep{PrMa:15,DeDi:18} to epidemic spreading \citep{GaBe:20}. The ultimate goal of being able to arbitrarily affect the behavior of these systems has spurred researchers across different scientific communities to investigate the controllability properties of linear complex networks \citep{PaZa:14,LoSo:19,YuZh:13}. In this framework, several works \citep{LiSl:11,LoGa:15} have revisited the classical tools of structural controllability \citep{Lin:74} from the viewpoint that in order to control complex networks, controllability must be guaranteed by a proper selection of the set of nodes (the driver node set in which control signals are injected. If the selection of the driver nodes ensures structural controllability, then the network will also be controllable in Kalman's sense for all possible edge weights but for a set of Lebesgue measure zero. Among the combinations of edge weights inside this set, there are those that induce the emergence of symmetries \citep{ChMe:14,ChMe:15} or equitable partitions \citep{GaFr:19} in the network graph. In the presence of symmetries, there exist permutations of the network nodes that leave the graph unchanged, and these symmetries induce a partition of the network in clusters. On the other hand, an equitable partition \citep{Godsil:97} clusters the network nodes such that the sum of the incoming edges in any node of the same cluster from nodes in any cluster is the same. While symmetries and equitable partitions cause loss of controllability \citep{AgGh:17}, they also induce the emergence of group consensus \citep{peso14,BlHu:19}, i.e., solutions in which the state of each node in the same cluster is the same. 

In this work we focus on networks with symmetries and we show that loss of controllability and emergence of group consensus are different sides of the same coin. Both are due to the presence of symmetry-induced invariant subspaces that are smaller than the entire state space. While these subspaces allow group consensus solutions to emerge, we also show that they encompass the network controllable subspace. Altoghether, our results show that while the dynamics orthognal to the group consensus subspace is not controllable, it is possible that the dynamics along this subspace can be controlled. If this is the case, control of the consensus solution can be achieved by designing controllers on a reduced network, whose nodes correspond to clusters of nodes of the original network, yielding a substantial computational advantage in the control design. 

Stabilizability of the dynamics orthogonal to the group consensus subspace is a necessary requirement to achieve group consensus, and is not guaranteed when the network dynamic matrix is not in the form of a Laplacian matrix, which is the case considered in this paper. Hence, in order to be able to stabilize the group consensus subspace, additional inputs must be added to the network. Here, we show how to perform a selection that allows independent design of the control action on the group consensus subspace and of the stabilizing action transverse to the subspace. We also give bounds on the number of indipendent inputs and on the number of nodes where these inputs must be injected, the \emph{drivers}, to achieve stabilizability of the group consensus subspace.

\section{Mathematical Preliminaries and Network Dynamics}
We denote by $\mathcal{G}(\mathcal{V},\mathcal{E})$ an undirected graph with $\mathcal{V} = \lbrace v_i, \ i=1,\dots,N\rbrace$ the set of $N$ nodes, and $\mathcal{E}\subseteq \mathcal{V}\times \mathcal{V}$ the set of edges defining the interconnections among the nodes. The symmetric binary matrix $A\in \mathbb{R}^{N\times N}$ is the adjacency matrix of the graph, that is, a matrix  whose elements are $A_{ij} = A_{ji} =1$ if $(i,j)\in \mathcal{E}$ and $A_{ij}= A_{ji}=0$ otherwise. A permutation $\pi(\mathcal{V}) = \widetilde{\mathcal{V}}$ is an automorphism (or symmetry) of $\mathcal{G}$ if (i) $\mathcal{V}=\widetilde{\mathcal{V}}$, i.e., $\pi$ does not add or remove nodes, and (ii) $(i,j)\in \mathcal{E}$, then $(\pi(i),\pi(j))\in \mathcal{E}$. The set of automorphisms of a graph with adjacency matrix $A$, with the operation composition, is the automorphism group, $aut(\G(A))$. Any permutation of this group can be represented by a permutation matrix $P$ that commutes with $A$, i.e., such that $PA=AP$. The set of all automorphisms in the group will only permute certain subsets of nodes (the \emph{orbits} or \emph{clusters}) among each other.  
For any two nodes in the same orbit there exists a permutation that maps them into each other.
Moreover, the \emph{coarsest orbital partition} is defined as the partition of the nodes corresponding to the orbits of the automorphism group. Given a partition $\Pi$ of the set $\mathcal{V}$ of the network nodes $\mathcal{V}$ into $s$ subsets $\lbrace S_1, S_2,... S_s\rbrace$, such that  $\cup_{i=1}^s S_i=\mathcal{V}$, $S_i \cap S_j =\emptyset $ for $i \neq j$, we can introduce the $N \times s$ indicator matrix $E^{\Pi}$, such that $E_{ij}^{\Pi}=1$ if node i belongs to $S_j$ and $E_{ij}^{\Pi}=0$ otherwise.

We consider a linear dynamical network described by
\begin{equation}\label{eq:contr_net}
\dot x = Ax + Bu.
\end{equation}
where $x\in \mathcal{X} = \mathbb{R}^N$ is the stack vector of the states of the $N$ network nodes and $u$ is the stack vector of the $M$ input signals injected in the network. Consistently, the $N\times N$ symmetric matrix $A$ defines the network topology, while the $N\times M$ matrix $B$ describes the way in which the $M$ input signals affect the network dynamics. Namely, if the $j$-th input is injected in the $i$-th node then $B_{ij} = 1$, while $B_{ij} = 0$ otherwise.
\section{Controllability Properties of Networks with Symmetries}\label{sec:contr_sym}
In this section, we will show how the presence of symmetries in the controlled network \eqref{eq:contr_net} affects controllability.
\begin{lemma}
The subset of automorphisms of $\G(A)$ given by the set of matrices $\mathcal{P}:=\lbrace P_i: \ P_iA = AP_i \ \mathrm{and} \  \ P_iB = B\rbrace$ forms a subgroup of $aut(\G(A))$. 
\end{lemma}
\begin{proof}
For the set $\mathcal{P}$ to be a subgroup, the following four properties must be true:
\begin{enumerate}[(i)]
\item $P_i(P_jP_k)=(P_iP_j)P_k$ $\forall$ $(P_i,P_j,P_k)$ $\in \mathcal{P}$;
\item $P_i\in \mathcal{P}$ is non singular $\forall \; i$;
\item $I\in \mathcal{P}$;
\item given any two matrices $P_i\in \mathcal{P}$ and $P_j\in \mathcal{P}$, then $P_iP_j\in \mathcal{P}$.
\end{enumerate}
Proving that the matrices in $\mathcal{P}$ satisfy property (i) and (ii) is trivial as (i) is true for any three square matrices with the same dimensions $(P_i,P_j,P_k) \in \mathcal{P}$ regardless of whether these are, or are not, in $\mathcal{P}$, while (ii) is true as permutation matrices are not singular.
Moreover, (iii) holds as $IA = AI = A$, and $IB = B$. Moreover, property (iv) is proved as 
$$(P_i P_j) A = P_i (P_j A)= P_i (A P_j)= A P_i P_j= A(P_i P_j)$$
from which we see that $P_iP_jA = AP_jP_i$ for all $(P_i,P_j) \in \mathcal{P}$. The proof is finally completed by noting that, as from our hypotheses $P_jB = P_iB = B$ for all $(P_i,P_j) \in \mathcal{P}$, it follows that $P_iP_jB = P_iB = B$.
\end{proof}

We will denote as $aut(\G(A,B))$ the group represented by the permutation matrices $P$ such that $PA-AP = 0$ and $PB-B = 0$. Similarly to $aut(\G(A))$, $aut(\G(A,B))$ partitions the set of network nodes into orbits or clusters, where an orbit is a subset of symmetric nodes. Hence, we can define the coarsest orbital partition $\Pi$ into clusters corresponding to the orbits of the automorphism group $aut(\G(A,B))$, $C_1, C_2, \dots, C_K$, such that  $\cup_{i=1}^K C_i = \mathcal{V}$, and $C_i\cap C_j =0$ for $i \neq j$. We will use the indicator matrix $E^{\Pi}$ to keep track of the orbit to which each node belongs. 
\begin{lemma}
Each orbit of the coarsest partition $\Pi$ induced by $aut(\G(A,B))$ is a subset of an orbit of the coarsest partition induced by $aut(\G(A))$. 
\end{lemma}
\begin{proof}
The thesis follows from the observation that if two (or more) nodes are permuted by a permutation matrix $P$ in $aut(\G(A,B))$ and thus belong to the same orbit, then they also belong to the same orbit of the coarsest orbital partition induced by $aut(\G(A))$, as the same matrix $P$ also belongs to $aut(\G(A))$.
\end{proof}

\begin{theorem}\label{thm:thm_1}
If there exists a permutation matrix $P \neq I$ such that $PA-AP =0$ and $PB-B=0$, then
\begin{itemize}
    \item[(i)] the set of states $\mathcal{X}_{or}:= \lbrace x: x_i = x_l \; \forall \, i,l \in \mathcal{C}_j, \ \forall j \rbrace\subset \mathcal{X}$, is an invariant subspace of the matrix $A$, i.e., $\forall x\in \mathcal{X}_{or}$, $Ax\in \mathcal{X}_{or}$;
    \item[(ii)] if $x_i=x_l$ then $\dot x_i = \dot x_l$ for all $(i,l) \in \C_j$ and for all $j$.
\end{itemize}
\end{theorem}
\begin{proof} Let us start by showing that if there exists a permutation matrix $P$ such that $PA=AP$ and $PB=B$, then the network state $x$ and the permuted state vector $y:=Px$ share the same dynamics. Indeed, by left multiplying both sides of eq. \eqref{eq:contr_net} by $P$ we get
$$P\dot x = PAx + PBu.$$
Then, as $PA=AP$ and $PB=B$, we get
$$\dot y = Ay + Bu.$$
Moreover, as there always exists a permutation matrix $P \in aut(\G(A,B))$ that maps  any two nodes belonging to the same clusters into each other \citep{KlPeSo:19}, then statement (ii) follows, i.e., nodes in the same clusters share the same dynamics, and thus if $x_i = x_l$ for all $i$ and $l$ in the same cluster, then also $\dot x_i =\dot x_l$. Moreover, this also ensures that the subspace made of all the points of the state-space such that $x_i = x_l$ $\forall$ $i,l\in \mathcal{C}_j$ and $\forall j = 1,\dots,K$ is $A$-invariant (statement (i)).
\end{proof}
Theorem \ref{thm:thm_1} establishes the existence of the group consensus subspace $\mathcal{X}_{or}$ for network \eqref{eq:contr_net}. Hence, to tackle consensus control problems, it is useful to introduce a transformation that allows us to separate the dynamics along $\mathcal{X}_{or}$ from that orthogonal to $\mathcal{X}_{or}$ itself. This task is accomplished by the Irreducible Representation (IRR) of the symmetry group through a transformation in a new coordinate system \citep{peso14} $z_{or}=T_{or}x$. The transformation matrix
\begin{equation}\label{eq:T_or}
T_{or}= \left[ \begin{array}{c}
T^{\parallel}\\ T^{\perp}
\end{array}\right]
\in \mathbb{R}^{N\times N}
\end{equation}
 is orthogonal, and the elements of the block $T^{\parallel}\in \mathbb{R}^{K\times N}$ are such that 
\begin{equation}\label{eq:T_or2}
T^{\parallel}_{ij}=\sqrt{\vert \C_i\vert }^{-1}
\end{equation}
if node $j$ is in cluster $i$ and 0 otherwise. Note that the $K$ rows of the matrix $T^{\parallel}$ are a basis of $\mathcal{X}_{or}$ while the rows of the matrix $T^{\perp}\in\mathbb{R}^{(N-K)\times N}$ are a basis of the orthogonal complement to the group consensus subspace. Notably, each of the rows of the matrix $T^{\perp}$, say the $j$-th, can be associated to a single cluster say $\mathcal{C}_i$. Namely, each element $T^{\perp}_{jl}$ is nonzero only if node $l$ belongs to the cluster $\mathcal{C}_i$. Consistently, the dynamic matrix $\tilde A = T_{or}AT_{or}^{T}$ has the following structure:
\begin{equation}\label{eq:transformation}
\tilde{A}=T_{or}AT_{or}^T=\left[
\begin{array}{cc}
A_{\parallel} & 0 \\ 
0 & A_{\perp}  \\ 
\end{array}\right].
\end{equation}
From eq. \eqref{eq:transformation}, we see that the IRR decouples the dynamics along the consensus subspace governed by the block $A_{\parallel}$ from that orthogonal to the group consensus subspace governed by the block $A_{\perp}$. In this new coordinate system, the dynamics of network \eqref{eq:contr_net} can be rewritten as
\begin{equation}\label{eq:trasformed_net}
\dot{z}_{or}= \tilde{A}z_{or}+\tilde{B}u,
\end{equation}
and 
\begin{equation}\label{eq:B_or}
\tilde{B}=T_{or}B=\left[
\begin{array}{c}
B_{\parallel}  \\ 
B_{\perp} 
\end{array}\right].
\end{equation}

Indeed, the pair $(A_{\parallel},B_{\parallel})$, which we  will denote as the \emph{quotient pair}, determines the controllability properties of the dynamics along the subspace $\mathcal{X}_{or}$ and thus our ability to control the consensus state, while the pair $(A_{\perp},B_{\perp})$ determines our ability to stabilize $\mathcal{X}_{or}$. We are interested in studying the controllability properties of the two pairs $(A_{\parallel},B_{\parallel})$ and $(A_{\perp},B_{\perp})$. Before doing so, we will present a few more details on this representation. First of all, let us point out that the block $T_{\parallel}$ of the matrix $T_{or}$ is such that $T_{\parallel} = {E^{\Pi}}^{\dagger}$, where $E^{\Pi}\in \mathbb{R}^{N\times K}$ is the indicator matrix corresponding to the coarsest orbital partition $\Pi$. Consistently, the state of the quotient network, the network associated to pair $(A_{\parallel},B_{\parallel})$, can be computed as $$z^{\parallel}= {E^{\Pi}}^{\dagger}x\in \mathbb{R}^{K}$$
 and thus we have that $A_{\parallel}= {E^{\Pi}}^{\dagger} A {E^{\Pi}}$ and $B_{\parallel}= {E^{\Pi}}^{\dagger} B$.

Now, we are ready to give the following theorem.
\begin{theorem}\label{thm:contr_sub}
If there exists a matrix $P\neq I$ such that $PA = AP$ and $PB=B$, then $\mathcal{X}_{or}$, the invariant subspace of the matrix $A$ associated to the cluster consensus solution, encompasses the controllable subspace.
\end{theorem}
\begin{proof}
To prove the statement we must show that if $PB=B$, $\mathcal{X}_{or}$ encompasses the range of $B$. Indeed, if $PB=B$, then $B$ is such that $b_{il}=b_{jl}$ for all $l$ and for all $i,j$ in the same cluster,  due to the fact that left-multiplying a vector by the matrix $P$ only permutes the elements associated to nodes of the same cluster. Hence, all the columns of $B$ and thus its range, are encompassed in $\mathcal{X}_{or}$. As the controllable subspace is defined as the smallest $A$-invariant subspace encompassing the range of $B$, the thesis follows.
\end{proof}
\begin{corollary}\label{cor:B_perp}
$B_{\perp} = \mathbf{0}_{(N-K)\times M}$.
\end{corollary}
\begin{proof}
The statement is a direct consequence of the statement of Theorem \ref{thm:contr_sub} and of the definition of $B_{\perp}$.
\end{proof}

\section{Controlling group consensus}\label{sec:gc_contr}
In Section \ref{sec:contr_sym} we have established some controllability limitations of networks with symmetries. Here, we show how to operate within these limitations in order to control group consensus.
\begin{corollary}\label{cor:optimal_control}
Consider a graph $\G (A,B)$ with coarsest orbital partition $\Pi$. If the pair $(A_{\parallel},B_{\parallel}$) is controllable, then for any cost function $J(u(t))$ the optimal control problem
\begin{subequations}\label{eq:orig_pr}
\begin{equation}\label{eq:orig_pr_a}
\min_u \ \int_0^{t_f} J(u(t))dt\\[-5mm]
\end{equation}
\begin{align}
    s.t. \quad
& \label{eq:orig_pr_b}
\dot x = Ax+Bu\\ 
& \label{eq:orig_pr_c}
x(0) = x_0\\
&\label{eq:orig_pr_d}
x(t_f) = x_f
\end{align}
\end{subequations}
admits solution $u^*(t):=\textit{argmin} \int_0^{t_f} J(u(t))dt $ if and only if $x_0$ and $x_f$ are such that $T_{or}x_0 = [z^{\parallel}_0 \ 0]^T$ and $T_{or}x_f = [z^{\parallel}_f \ 0]^T$, i.e., $z^{\perp}_0=z^{\perp}_f = 0$. Moreover, if $T_{or}x_f = [z^{\parallel}_f \ 0]$, then $u^*= u^{**}$, where $u^{**}$ is the solution of the following optimal control problem
\begin{subequations}\label{eq:quo_pro}
\begin{equation}\label{eq:quo_pro_a}
\min_{u} \int_0^{t_f} J(u(t))dt\\[-5mm]
\end{equation}
\begin{align}
    s.t. \quad
& \label{eq:quo_pro_b}
\dot{z}^{\parallel} = A_{\parallel}z^{\parallel}+B_{\parallel}u\\
&\label{eq:quo_pro_c}
z^{\parallel}(0) = T^{\parallel}x_0\\
&\label{eq:quo_pro_d}
z^{\parallel}(t_f) = T_{\parallel} x_f. \end{align}
\end{subequations}
\end{corollary}
\begin{proof}
From Theorem \ref{thm:contr_sub}, if $x_f$ is such that $z^{\perp}_f \neq 0$ then $x_f$ is not reachable, while if $x_0$ is such that $z^{\perp}_0\neq 0$ then $x_f$ is not reachable from $x_0$. Hence, in both cases problem \eqref{eq:orig_pr} is not feasible. On the other hand, if $x_0$ and $x_f$ are such that $z^{\perp}_0=z^{\perp}_f = 0$, then both $x_0$ and $x_f$ belong to $\mathcal{X}_{or}$, which we know coincides with the controllable subspace from Theorem \ref{thm:contr_sub} and from the hypotheses. Then, reaching $z^{\parallel}_f$ is equivalent to reaching the point $x_f$. Hence, to prove our thesis, we are left with showing that $u^* = u^{**}$. We will do so by showing that problems
\eqref{eq:orig_pr} and \eqref{eq:quo_pro} share the same decision variables, cost function, and constraints. Indeed, the decision variables are the same by definition, as well as the cost function as input signals are not affected by equivalent transformations. Finally, to prove that problems \eqref{eq:orig_pr} and \eqref{eq:quo_pro} share the same constraints, let us show that by left multiplying both sides of equations \eqref{eq:orig_pr_b}-\eqref{eq:orig_pr_d}, we obtain eqs. \eqref{eq:quo_pro_b}-\eqref{eq:quo_pro_d} together with a set of equations that are always verified independently of $u$. Indeed from the hypotheses this is true for eqs. \eqref{eq:orig_pr_c} \eqref{eq:orig_pr_d}, as left multiplying both by the matrix $T_{or}$ we obtain eqs. \eqref{eq:quo_pro_c} and \eqref{eq:quo_pro_d} together with two sets of $N-K$ equations of the type $0=0$. Finally, from eqs. \eqref{eq:transformation}, \eqref{eq:B_or}, and Corollary \ref{cor:B_perp} we know that left-multiplying eq. \eqref{eq:orig_pr_b} by $T_{or}$ yields
\begin{subequations}
\begin{eqnarray}\label{eq:dyn_a}
\dot z^{\parallel} &=& A_{\parallel}z^{\parallel}+ B_{\parallel}u
\\
\label{eq:dyn_b}
\dot z^{\perp} &=& A_{\perp}z^{\perp}.    
\end{eqnarray}
\end{subequations}
As $z^{\perp}(0)=0$, from eq. \eqref{eq:dyn_b} we have that $z^{\perp}(t)=0$ for all $t$, and thus eq. \eqref{eq:dyn_a}, which coincides with eq. \eqref{eq:quo_pro_b}, captures completely the dynamics in eq. \eqref{eq:orig_pr_b} independently of $u$. Hence, problem \eqref{eq:orig_pr} and the reduced order problem in \eqref{eq:quo_pro} share the same decision variables, cost function, and constraints which implies that $u^*=u^{**}$.
\end{proof}
\begin{remark}
Note that Corollary \ref{cor:optimal_control} provides an approach to design an input to control group consensus. A viable alternative is to solve
\begin{subequations}\label{eq:out_pro}
\begin{equation}\label{eq:out_pro_a}
\min_{u} \int_0^{t_f} J(u(t))dt\\[-5mm]
\end{equation}
\begin{align}
    s.t. \quad
& \label{eq:out_pro_b}
\dot {x}=Ax+Bu
\\
& \label{eq:out_pro_b_bis}
y = {E^{\Pi}}^T x
\\
& \label{eq:out_pro_c}
x(0) = x_0 \in \mathcal{X}_{or}
\\
&\label{eq:out_pro_d}
y(t_f) = y_{f}.    
\end{align}
\end{subequations}
with $E^{\Pi}$ being the indicator matrix corresponding to the partition $\Pi$ of the network nodes, and 
$$\frac{y_i}{|\mathcal{C}_i|}$$
being the consensus value for all the nodes of the cluster $\mathcal{C}_i$.
\end{remark}
\begin{remark}\label{rem:stab}
Corollary \ref{cor:optimal_control} provides an approach to control the consensus solution. The stability properties of the group consensus subspace are determined by the eigenvalues of the block $A_{\perp}$ of the matrix $\tilde A$ in eq. \eqref{eq:transformation}. However this solution is not stabilizable, as the dynamics orthogonal to the group consensus subspace are uncontrollable (see Theorem \ref{thm:contr_sub}).  
\end{remark}

Motivated by the considerations in Remark \ref{rem:stab}, we now tackle the problem of selecting a set of nodes in which additional inputs must be injected to stabilizable $\mathcal{X}_{or}$. To do so, we leverage the following conditions from \citep{PBH:70}.
\begin{defi}\label{def:PBH}
Given a pair $(A,B)$ an eigenvalue $\lambda_i$ of $A$ is controllable if and only if $\exists j$ such that $v_i^{T}b_j\neq 0$, for any eigenvector $v_i$ associated to $\lambda_i$.
\end{defi}
\begin{theorem}\label{thm:PBH} \citep{PBH:70}
A dynamical system defined by the pair $(A,B)$ is stabilizable if and only if every unstable eigenvalue of $A$ is controllable.
\end{theorem}

We denote by $w$ the $W$-dimensional vector of the additional inputs  and by $D$ the $N\times W$ dimensional matrix indicating the nodes in which these inputs are injected, that is, the drivers. Namely, $D_{ij}\neq 0$ if the $j$-th additional input $w_j$ is injected in the $i$-th network node and $0$ otherwise. Considering these additional inputs leads to rewriting the dynamics of the network in eq. \eqref{eq:contr_net} as
\begin{equation}\label{eq:contr_net_b}
\dot {x} = Ax+Bu+Dw.
\end{equation}
As a result, applying the transformation $T_{or}$ in eq. \eqref{eq:transformation} to the controlled network in eq. \eqref{eq:contr_net_b} yields
\begin{equation}\label{eq:transformed_net_b}
\dot{z} = \tilde A z +\tilde B u + \tilde D w,   
\end{equation}
where
\begin{equation}\label{eq:tilde_D}
    \tilde D = T_{or}D = \left[ \begin{array}{c}
         D_{\parallel}  \\
         D_{\perp}
    \end{array}\right].
\end{equation}
We constrain the selection of the matrix $D$ to be such that the input signals $w$ do not affect the dynamics along the group consensus subspace, so to allow independent design of (i) the control action $u$ responsible for controlling the group consensus solution and (ii) the stabilizing action $w$.

To be able to formulate and solve our driver node selection problem, let us relabel the eigenvalues of $A$ so that the first $K$ are also eigenvalues of $A_\parallel$ and the last $(N-K)$ are also eigenvalues of $A_\perp$  (here we just list all the eigenvalues of $A$ regardless of their multiplicity). Note that this is possible from the block diagonal structure of $\tilde A$ in eq. \eqref{eq:transformation}.
 After this relabeling, the eigenvectors of $A$ associated with its first $K$ eigenvalues span the group consensus subspace, while the eigenvectors of $A$ associated with the last $(N-K)$ eigenvalues span its orthogonal complement. In particular, the last $(N-K)$ eigenvalues of $A$ determine the stability properties of the group consensus subspace. Moreover, we denote by $\Omega_i$ the subspace of the eigenspace of the eigenvalue $\lambda_i$ of $A$ that is orthogonal to $\mathcal{X}_{or}$ and by $\mu_i$ the dimension of $\Omega_i$. Given a vector $d$, we denote by $proj_{\Omega_i}(d)$ its projection on $\Omega_i$. Finally, we denote by $\Lambda^{\perp}$ the subset of the eigenvalues of $A$ with nonnegative real part that are also eigenvalues of $A_{\perp}$. Thanks to these preliminary considerations and notation, we can now formulate our driver node selection problem
\newline \underline{\textbf{Problem 1:}} Select a matrix $D$ such that
\begin{subequations}
\begin{gather}\label{eq:D_par_0}
D_{\parallel} =0 \\
\label{eq:A_perp_stabilizable}
(A_{\perp},D_{\perp}) \ \mathrm{is \ stabilizable}
\end{gather}
\end{subequations}

\begin{algorithm}[h!]\caption{ Driver Node Selection Algorithm
	 \label{algo:main_algo}}
	\begin{algorithmic}
		\Procedure{Initialization }{$i=1$, $D$ is the empty matrix, $j=0$}
		\While{$i\leq |\Lambda^{\perp}|$}	
		    \State $\Delta_i = \lbrace D_j: proj_{\Omega_i}(D_i) \neq 0 \ \land \ \nexists    D_k : proj_{\Omega_i}(D_k) \parallel proj_{\Omega_i}(D_j)\rbrace $
			\State $h_i =|\Delta_i|$
			\While{$j\leq \mu_i-h_i$}
			\State $j=j+1$
			\State Build an $N$-dimensional vector $D_j$ by solving
			\vspace*{-4mm}
			\begin{align}\label{eq:nonzero_proj} 
			    &proj_{\Omega_i}(D_j) \neq 0\\ \label{eq:nonparalle_proj}
			    &proj_{\Omega_i}(D_j) \neq proj_{\Omega_i}(D_m) \  \forall m< j\\\label{eq:parallel_condition}
			    &\sum_{k\in \mathcal{C}_l} D_{j}(k) = 0 \ \forall l
			\end{align}\vspace*{-6mm}
			\State $D = [D \ D_j]$
			\EndWhile
		\State $i = i+1$
		\EndWhile
		\EndProcedure
	\end{algorithmic}
\end{algorithm}

Algorithm 1 prescribes to initialize the matrix $D$ as an empty matrix. Then, for all the eigenvalues in the set $\Lambda^{\perp}$, we find the number $h_i$ of columns of the matrix $D$ with nonzero and linearly independent projection on $\Omega_i$, that is, the subspace of the eigenspace associated to $\lambda_i$ that is orthogonal to the group consensus subspace. Then, we add $\mu_i-h_i$ column vectors to the matrix $D$ each having non-zero and linearly independent projection on $\Omega_i$, thus ensuring, from Definition \ref{def:PBH} that $\lambda_i$ is controllable. Thanks to the condition in eq. \eqref{eq:parallel_condition}, these $\mu_i-h_i$ added columns will be orthogonal to the group consensus subspace thus ensuring $D_{\parallel}=0$. Doing so for all $\lambda_i$ in $\Lambda^{\perp}$ ensures the pair $(A_{\perp},D_{\perp})$ is stabilizable thanks to Theorem \ref{thm:PBH}.
\begin{theorem}\label{thm:alg_ok}
Algorithm 1 solves Problem 1.
\end{theorem}
\begin{proof}
To prove that any matrix selected by Algorithm 1 satisfies condition \eqref{eq:D_par_0} it suffices to note that from eq. \eqref{eq:tilde_D} and the structure of the matrix $T_{or}$ in eq. \eqref{eq:T_or} we have that the $i$-th element of the $j$-th column of $D_{\parallel}$ is obtained as $\sum_{k\in \mathcal{C}_i}D_j(k)$. Then, eq. \eqref{eq:D_par_0} follows directly from eq. \eqref{eq:parallel_condition}. 
On the other hand, note that from Theorem \ref{thm:PBH} and Definition \ref{def:PBH}, to prove that any matrix selected according to Algorithm 1 satisfies \eqref{eq:A_perp_stabilizable} it suffices to show that for each eigenvector, say $v^{\perp}_j$ of $A_{\perp}$ associated to an eigenvalue that is encompassed in the set $\Lambda^{\perp}$ there exists a column $D^{\perp}_l$ of the matrix $D_{\perp}$ such that ${v^{\perp}_j}^T D^{\perp}_l \neq 0$. In turn, as any $\mu_i$ vectors of $\Omega_i$ can be chosen as eigenvectors of $A_{\perp}$, and as the columns of $D_{\perp}$ are the projection of the columns of $D$ on the orthogonal complement to the group consensus subspace, ensuring that for any $v^{\perp}_j$ associated to an eigenvalue $\lambda_i\in \Lambda^{\perp}$ there exists $D^{\perp}_l$ such that ${v^{\perp}_j}^T D^{\perp}_l \neq 0$ is equivalent to ensuring that there exist $\mu_i$ columns of $D$ that span $\Omega_i$. As this is ensured by the inner while loop in Algorithm 1 thanks to eqs. \eqref{eq:nonzero_proj} and \eqref{eq:nonparalle_proj}, the thesis follows.
\end{proof}
\begin{remark}
Note that while indeed the symmetries of the pair $(A,[B \  D])$, with $D$ selected according to Algorithm 1, are not the same of that of the pair $(A,B)$, this has no effect on the dynamics along the group consensus manifold as from Problem 1 and Theorem \ref{thm:alg_ok} we know that $D_{\parallel}=0$. Consistently, as the control signal $w$ is conceived to be a stabilizing feedback action, it will vanish asymptotically, and in the absence of perturbations the network dynamics will revert to that in eq. \eqref{eq:contr_net}.
\end{remark}
\begin{corollary} \label{cor:min_inputs}
The number of independent input signals required to solve Problem 1 is lower bounded by $$\max_{i:\lambda_i \in \Lambda^{\perp}} |\Omega_i|.$$
\end{corollary}
\begin{proof}
Let us start by noting that any vector in $\Omega_i$ is an eigenvector of $A_{\perp}$ associated to $\lambda_i$. Hence, for the stabilizability condition in Theorem \ref{thm:PBH} to be verified for the pair $(A_{\perp},D_{\perp})$, we must have that for all $\lambda_i \in \Lambda^{\perp}$ there exist $|\Omega_i|$ columns of $D^{\perp}$, and thus also of $D$, with nonzero and non-parallel projection on $\Omega_i$. Hence, the pair $(A_{\perp},D_{\perp})$ can be stabilizable only if the number of columns of $D$ is at least equal to $\max_{i:\lambda_i \in \Lambda^{\perp}} |\Omega_i|$ which proves our statement.
\end{proof}
After giving a bound on the number of input signals required to solve Problem 1, let us now give a bound on the number of drivers, i.e., the number of rows of $D$ encompassing at least a nonzero entry, required to solve Problem 1. To do so, let us define the operator 
\begin{equation*}
    |\cdot|_{\emptyset}:=
    \begin{cases}
    |\cdot| \ \mathrm{if} \ |\cdot|>0\\
    -1 \ \mathrm{otherwise}
    \end{cases}
\end{equation*}
\begin{corollary} \label{cor:min_drivers}
The number of rows of the matrix $D$ with at least one nonzero entry required to solve Problem 1 is lower bounded by 
$$\max_{i:\lambda_i \in \Lambda^{\perp}} |\Omega_i|_{\emptyset} + 1.$$
\end{corollary}
\begin{proof}
From corollary \ref{cor:min_inputs}, we know that the number of columns of $D$ required to stabilize $\mathcal{X}_{or}$ is lower bounded by $\max_{i:\lambda_i \in \Lambda^{\perp}} |\Omega_i|$. As the projections of these columns on $\Omega_{i^*}$, with $i^*=\mathrm{argmax}_{i:\lambda_i \in \Lambda^{\perp}}|\Omega_i|$, must be nonzero and non parallel, then the rank of the matrix $D$ is lower bounded by $\max_{i:\lambda_i \in \Lambda^{\perp}} |\Omega_i|$. On the other hand, to ensure the condition in \eqref{eq:D_par_0} is fulfilled, each column of $D$ must be parallel to $\mathcal{X}_{or}$ which is true iff the columns of $D$ verify eq. \eqref{eq:parallel_condition}, that is, their elements sum to zero. Hence, for the matrix $D$ to be zero column sum and have at least rank $\max_{i:\lambda_i \in \Lambda^{\perp}} |\Omega_i|$ it must have at least $\max_{i:\lambda_i \in \Lambda^{\perp}} |\Omega_i|_{\emptyset}+1$ rows encompassing nonzero entries thus proving our statement. 
\end{proof}
Corollary \ref{cor:min_drivers} provides a bound on the number of driver nodes required to solve Problem 1. We will now show how to exploit the clusters induced by the network symmetries to give a different bound from that provided in Corollary \ref{cor:min_drivers}. To do so, let us denote by $\Omega_i^j$ the subspace of $\Omega_i$ that is spanned by vectors $e_l$, $l=1,\dots,|\Omega_i^j|$ such that each element $e_{lm}$ of $e_l$ is nonzero  iff node $m$ is encompassed in cluster $\mathcal{C}^j$. Roughly speaking, $\Omega_{i}^j$ is the $j$-th cluster specific subspace of $\Omega_i$. As in general $\Omega_i$ cannot be completely spanned by cluster specific vectors, we have that $\Omega_i = \cup_{j=1}^K \Omega_i^j  + \tilde \Omega_i$, where $\tilde \Omega_i$ is thus the subspace of $\Omega_i$ that cannot be spanned by cluster specific vectors. Finally let us relabel the network nodes so that node $i$ belongs to $\mathcal{C}_j$ if $|\mathcal{C}_{j-1}|<i\leq |\mathcal{C}_{j}|$, with $|\mathcal{C}_{0}|=0$ as $\mathcal{C}_{0}$ does not exist. Then, the matrix $D$ can be decomposed in blocks as follows
\begin{equation}
D= \left[
\begin{array}{c}
D^1  \\
D^2\\
\vdots\\
D^K
\end{array}
\right]    
\end{equation}
with each $D^j$ having $|\mathcal{C}_j|$ rows.
\begin{corollary} \label{cor:min_drivers_b}
The number of rows of the matrix $D$  encompassing nonzero entries required to solve Problem 1 is lower bounded by 
\begin{equation}\label{eq:bound_on_drivers_b}
    \sum_{j=1}^{K}\left( \max_{i:\lambda_i \in \Lambda^{\perp}} |\Omega_i^j|_{\emptyset} + 1\right).
\end{equation}
\end{corollary}
\begin{proof}

From Theorem \ref{thm:PBH}, Definition \ref{def:PBH}, and eq. \eqref{eq:D_par_0}, we know that to solve Problem 1 we need to ensure that each $\lambda_i \in \Lambda^{\perp}$ is made controllable by a matrix $D$ such that $\sum_{l\in \mathcal{C}_j}D_{li} = 0$ $\forall i$. Moreover, from Corollary \ref{cor:min_inputs} and as $\Omega_i^j$ is spanned by cluster specific vectors, it is possible to show that to ensure $\lambda_i\in \Lambda^{\perp}$ is controllable we need that at least $|\Omega_i^j|$ columns of the matrix $D^j$ have nonzero and non parallel projection on $\Omega_i^j$. Hence, these columns must define a matrix that is full rank but also zero column sum so to ensure fulfillment of eq. \eqref{eq:D_par_0}. This implies that stabilizing any $\lambda_i \in \Lambda^{\perp}$ requires that at least $|\Omega_i^j|+1$ rows of $D^j$ encompass a nonzero entry for all $j$ such that $\Omega_i^j \neq \emptyset$, and thus the total number of rows of the matrix $D$ encompassing a nonzero entry is lower bounded by the quantity in \eqref{eq:bound_on_drivers_b}.
\end{proof}

\begin{remark}\label{rem:c_spec_eigs}
The problem of identifying the cluster specific vectors spanning the subspaces $\Omega_i^j$ for all $i$ and $j$ can be easily solved using the IRR transformation $T_{or}$. Indeed, one of the properties of this transformation is to have cluster specific rows that can be linearly combined through the coefficients of the eigenvectors of the corresponding block of $\tilde A$ to generate eigenvectors of $A$. Therefore, each eigenvector of $A$ associated to an eigenvalue $\lambda_i$ obtained through this procedure either belongs to (i) $\Omega_i^j$ if the rows that are combined to obtain them are all associated to the same cluster $\mathcal{C}_j$, or (ii) $\tilde \Omega_i$ otherwise.
\end{remark}

\section{Numerical example}

We consider the $N=8$ node network in Fig. \ref{fig:example_net}. A study of the symmetries of the pair $(A,B)$ shows that there are $K=3$ orbital clusters, $\C_1\cup \C_2 \cup \C_3= \mathcal{V}$ and $\C_1 = \left\lbrace 1,2,3,4\right\rbrace$, $\C_2= \left\lbrace 5, 6 \right\rbrace$, $\C_3= \left\lbrace 7, 8 \right\rbrace$. The corresponding indicator matrix is 
\begin{equation}\label{eq:ind_matrix}
{E^{\Pi}}^T= \left[\begin{array}{cccccccc}
1 & 1 & 1 & 1& 0&0 & 0 &0 \\
0 & 0 & 0 & 0& 1&1 & 0 &0 \\
0 & 0 & 0 & 0& 0&0 & 1 &1 
\end{array}\right].   
\end{equation}
We tackle the problem of steering the network state towards the group consensus value $[\mathbf{1}_{1 \times 4} \;  \mathbf{2}_{1\times 2} \;  \mathbf{3}_{1\times 2}]^T$ in $t_f = 5$ seconds. To do so, according to the results in Section \ref{sec:gc_contr} we must first decouple the dynamics along and transverse to the group consensus subspace by leveraging the state transformation $z=T_{or}x$ with
\begin{equation}\label{eq:}
T_{or}=\setlength{\arraycolsep}{1pt}
\left[ \begin{array}{cccccccc}
0.5    &0.5  &0.5  &0.5   &0    &0  &0   &0\\
0  &0    &0  &0    &0  &0    &\sqrt{2}^{-1}  &\sqrt{2}^{-1}\\
0 &0 &0   &0    &\sqrt{2}^{-1}   &\sqrt{2}^{-1} &0   &0\\
0.5  &0.5    &-0.5 &-0.5    &0  &0    &0 &0\\
0   &0 &0   &0 &\sqrt{2}^{-1} &-\sqrt{2}^{-1} &0 &0\\
\sqrt{2}^{-1}   &-\sqrt{2}^{-1}  &0  &0  &0  &0  &0  &0\\
0   &0    &\sqrt{2}^{-1}   &-\sqrt{2}^{-1} &0   &0 &0   &0\\
0  &0    &0  &0    &0  &0  &\sqrt{2}^{-1}    &-\sqrt{2}^{-1}
    \end{array}\right],
\end{equation}
obtaining
\begin{equation}\label{eq:ex_mats}
\setlength{\arraycolsep}{2.5pt}
    A_{\parallel} = \left[
    \begin{array}{ccc}
         0 &0 &\sqrt{2}\\
         0 &0 &2\\
         \sqrt{2} &2 &0
    \end{array}
    \right]\!,\  
    B_{\parallel} = \left[
    \begin{array}{c}
    0     \\
    \sqrt{2}\\
    0
    \end{array}
    \right]\!,\ 
    A_{\perp} = \left[
    \begin{array}{ccccc}
         0&-\sqrt{2}&0&0&0\\
         -\sqrt{2}&0&0&0&0\\
         0&0&0&0&0\\
         0&0&0&0&0\\
         0&0&0&0&0
    \end{array}
    \right]\!,\ 
    B_{\perp} = \left[
    \begin{array}{c}
    0\\    0\\    0\\ 0\\ 0
    \end{array}
    \right]\!.
\end{equation}
Consistently with Corollary \ref{cor:B_perp}, we obtain that $B_{\perp}=0$. Moreover, the pair $(A_{\parallel},B_{\parallel})$
defines the dynamics of the quotient network, whose three node structure is portrayed in Fig. \ref{fig:example_net}. As the reader may easily check, the pair $(A_{\parallel},B_{\parallel})$ is controllable, and thus to control the dynamics along $\mathcal{X}_{or}$ we pose the following minimum energy control problem:
\begin{figure}
    \centering
	\begin{minipage}{0.35\columnwidth}
	\resizebox{\textwidth}{!}{	
	\begin{tikzpicture}[auto, on grid, semithick, style_gen/.style = {circle, ultra thick, minimum size =1.5cm, inner sep=1pt}, 
	]	
	\node[style_gen, fill = red, draw=red, line width=0.3cm,] at (7.50, 10.00)(A1) {\Huge 2};
	\node[style_gen, fill = red, draw=red, line width=0.3cm,] at (2.50,10.00)(A2) {\Huge 1};
	\node[style_gen, fill = red, draw=red, line width=0.3cm,] at (7.50,0)(A3) {\Huge 4};
	\node[style_gen, fill = red, draw=red, line width=0.3cm,] at (2.5,0)(A4) {\Huge 3};
	\node[style_gen, fill = cyan, draw=cyan, line width=0.3cm,] at (7.5,5.00)(A5) {\Huge 8};
	\node[style_gen, fill = cyan, draw=cyan, line width=0.3cm,] at (2.5,5.00)(A6) {\Huge 7};
	\node[style_gen, fill = yellow, draw=yellow, line width=0.3cm,] at (5.00,7.50)(A7) {\Huge 5};
	\node[style_gen, fill = yellow, draw=yellow, line width=0.3cm,] at (5.00,2.50)(A8) {\Huge 6};

	\begin{scope}[on background layer]
	\draw(A1) -- (A7);
	\draw(A2) -- (A7);
	\draw(A7) -- (A6);
	\draw(A7) -- (A5);
	\draw(A8) -- (A6);
	\draw(A8) -- (A5);
	\draw(A8) -- (A3);
	\draw(A8) -- (A4);
   \draw[->, >=stealth] (2.5,7) node[above]{\LARGE $u$} to (A6);
   \draw[->, >=stealth] (7.5,7) node[above]{\LARGE $u$} to (A5);

	\end{scope}
	\end{tikzpicture}
}
\end{minipage}
\hspace*{2cm}
\begin{minipage}{.45\columnwidth}
\resizebox{\textwidth}{!}{	
$
\setlength{\arraycolsep}{5pt}
	A = \left[
	\begin{array}{cccccccc}
		0 & 0 & 0 & 0 & 1 & 0 & 0 & 0 
		\\ 
		0 & 0 & 0 & 0 & 1 & 0 & 0 & 0 
		\\ 
		0 & 0 & 0 & 0 & 0 & 1 & 0 & 0 
		\\ 
		0 & 0 & 0 & 0 & 0 & 1 & 0 & 0
		\\ 
		1 & 1 & 0 & 0 & 0 & 0 & 1 & 1
		\\
		0 & 0 & 1 & 1 & 0 & 0 & 1 & 1
		\\
		0 & 0 & 0 & 0 & 1 & 1 & 0 & 0
		\\
		0 & 0 & 0 & 0 & 1 & 1 & 0 & 0\\
	\end{array}\right]
$\phantom{$^T$}}
\\[10pt]
\resizebox{\textwidth}{!}{	
\setlength{\arraycolsep}{4.6pt}
$
	B = \left[ \begin{array}{cccccccc}
		0 &0 & 0& 0& 0& 0& 1& 1 
	\end{array}\right]^T
	$
}\\[1cm]
\resizebox{\textwidth}{!}{
\begin{tikzpicture}[auto, on grid, semithick, style_gen/.style = {circle, ultra thick, minimum size =2.3cm, inner sep=1pt}, 
	]	
	\node[style_gen, fill = red, draw=red, line width=0.3cm,] at (0, 0.00)(C1) {\Huge $\mathcal{C}_1$};
	\node[style_gen, fill = yellow, draw=yellow, line width=0.3cm,] at (5,0.00)(C2) {\Huge $\mathcal{C}_2$};
	\node[style_gen, fill = cyan, draw=cyan, line width=0.3cm,] at (10, 0.00)(C3) {\Huge $\mathcal{C}_3$};
	\draw(C1) -- node[above]{\Huge $\sqrt{2}$} (C2) -- node[above]{\Huge $2$} (C3);
   \draw[->, >=stealth] (10,3) node[above]{\Huge $u$} to (C3);
\end{tikzpicture}
}
\end{minipage}
\caption{A simple 8 node network, with edge weights all equal to one. The coarsest orbital partition of the network shown in the figure has three clusters $\C_1$, $\C_2$, and $\C_3$, colored in red, yellow, and cyan respectively.}
    \label{fig:example_net}
\end{figure}
\begin{equation}\label{eq:opt_quotient}
\begin{aligned}
\min_{u} &\dfrac{1}{2}\int_0^{5} u^{T}(t)u(t)dt\\
s.t.\\
\dot z^{\parallel} = &A_{\parallel}z^{\parallel}+B_{\parallel}u\\
z^{\parallel}(0) = & T^{\parallel} x_0 \\
z^{\parallel}(5) = &T^{\parallel}[\mathbf{1}_{1 \times 4} \;  \mathbf{2}_{1\times 2} \;  \mathbf{3}_{1\times 2}]^T = [2\;\; 3\sqrt{2}\;\; 2\sqrt{2}]^T
\end{aligned}
\end{equation}
where $z^{\parallel}\in \mathbb{R}^3$ is the state vector of the quotient network.

The solution of this optimal control problem is
\begin{equation}\label{eq:opt_input_value}
\begin{aligned}
u^{**}(t)= &B_{\parallel}^Te^ {A_{\parallel}^T(5-t)}W^{-1}(z^{\parallel}(5)-e^{5A_{\parallel}}z(0))\\
=& B_{\parallel}^T(V_{\parallel}^T)^{-1}e^{\Lambda_{\parallel}(5-t)}V_{\parallel}^TW^{-1}(z^{\parallel}(5)-V_{\parallel}^{-1}e^{5\Lambda_{\parallel}}V_{\parallel}z(0))\\
\approx & -0.00003 e^{\sqrt{6}(5-t)}+2.54 e^{-\sqrt{6}(5-t)}+0.732
\end{aligned}
\end{equation}
where
\begin{equation}\label{eq:gram}
W(t_0,t_f)= \int_{t_0}^{t_f} e^{A_{\parallel}(t_f-t)}B_{\parallel}B_{\parallel}^Te^{A_{\parallel}^T(t_f-t)}dt
\end{equation}
is the reachability gramian of the quotient network. Note that the optimal control input is a linear combination of the three eigenmodes corresponding to the three clusters of the orbital partition $\Pi$ of $\mathcal{G}(A,B)$. It's worth underlining that, since the consensus subspace is unstable, numerical computation of the optimal control solution is hard due to the positive eigenvalue $\sqrt{6}$. Notably, due to the low dimensionality of the quotient network, the IRR allows us to solve \eqref{eq:opt_quotient} analytically.

Having dealt with controlling the dynamics along the group consensus subspace, we can now turn to stabilizing the dynamics transverse to this subspace. To this aim, note that the spectrum of the matrix $A_{\perp}$ in \eqref{eq:ex_mats} is composed of the following set of eigenvalues
\begin{equation}
    \lbrace -\sqrt{2}, \ 0, \ \sqrt{2} \rbrace
\end{equation}
with the geometric multiplicity of the null eigenvalue being equal to $3$, and the other two eigenvalues being simple. Hence, in order to apply Algorithm 1, we must first consider that $\Lambda^{\perp}=\lbrace\ 0, \ \sqrt{2}\rbrace$, with $\mu_1 = 3$, and $\mu_2 = 1$. Then, setting $i=1$, and as $D$ is initialized as the empty matrix, then $h_1=0$ as $\Delta$ is the empty set and we can enter the inner while loop. The three vectors spanning $\Omega_1$ are the last three rows of the matrix $T_{or}$ that brings the system in the IRR-coordinates, namely

\small
\begin{equation*}
\setlength{\arraycolsep}{4pt}
\left[ \begin{array}{cccccccc}
\sqrt{2}^{-1}   &-\sqrt{2}^{-1}  &0  &0  &0  &0  &0  &0\\
0   &0    &\sqrt{2}^{-1}   &-\sqrt{2}^{-1} &0   &0 &0   &0\\
0  &0    &0  &0    &0  &0  &\sqrt{2}^{-1}    &-\sqrt{2}^{-1}
    \end{array}\right]^T
\end{equation*}
\normalsize
and a feasible solution that iteratively solves eqs. \eqref{eq:nonzero_proj}-\eqref{eq:parallel_condition} is
\begin{equation}\label{eq:first_step}
    \begin{array}{rcrccccccccl}
    D_1 &=& \big[\hspace*{-2mm}& 1& 0& 0& -1& 0& 0& 0& 0 &\hspace*{-2mm}\big]^T, \\ 
    D_2 &=& \big[\hspace*{-2mm}&0& 0& 0& 0& 0& 0& 1& -1  &\hspace*{-2mm}\big]^T, \\ 
    D_3 &=& \big[\hspace*{-2mm}&0& 0& -1& 1& 0& 0& 0& 0  &\hspace*{-2mm}\big]^T.
    \end{array}
\end{equation}
Hence, we can turn to $i=2$ noting that as the vector 
$$
\left[
\begin{array}{cccccccc}
-0.35 & -0.35 & 0.35 & 0.35 & -0.50 & 0.50 & 0 & 0
\end{array}
\right]^T$$
is a basis for $\Omega_2$, then $h_2=1$ as there already exists a column of $D$, namely $D_1$ in eq. \eqref{eq:first_step} with nonzero projection on $\Omega_2$. Hence, as $\mu_2 = 1$, and $|\Lambda^{\perp}|=2$, the driver node selection procedure comes to an end. Note that this solution achieves both the bound given in Corollary \ref{cor:min_inputs} as well as that given in Corollary \ref{cor:min_drivers_b} and thus minimizes both the number of input signals and the number of driver nodes required to stabilize $\mathcal{X}_{or}$. 

Having performed the selection of the matrix $D$ that ensures stabilizability of the pair $(A_{\perp},D_{\perp})$ we can now turn our attention to designing the stabilizing signal $w$ as
$$w=-Gz_{\perp}$$
with the matrix $G$ being such that the eigenvalues of the matrix $(A_{\perp}-D_{\perp}G)$ are all smaller than or equal to $-\sqrt{2}$, the only negative eigenvalue of $A_{\perp}$ which we do not move. Specifically, we design $G$ so that all the originally nonnegative eigenvalues are placed in $-2$. This selection ensures that the slowest time constant of the transverse dynamics is the one of the only stable eigenvalue we did not touch ($1/\sqrt{2}$). Note that this placement ensures the transverse dynamics become negligible well before the time $t_f=5$ when the dynamics parallel to the group consensus subspace will converge to the target state $z^{\parallel}(5)$.
The designed control inputs can be now used to steer the network towards the group consensus state $[\mathbf{1}_{1\times 4} \ \mathbf{2}_{1\times 2} \ \mathbf{3}_{1\times 2}]$. In Figure \ref{fig:state_ev} we report the network state evolution (panel a) and the control inputs (panel b). As expected, the optimal control input $u^{**}$ in eq. \eqref{eq:opt_input_value}, shown in black in Figure \ref{fig:state_ev}(b) is able to steer the nodes in $\C_1$ to $1$, the nodes in $\C_2$ to $2$ and the nodes in $\C_3$ to $3$ at $t_f=5$. In the meantime, the stabilizing control input $w$ makes the transverse clustered synchronous solution stable, ensuring the network state converges on the cluster consensus subspace. Note that as expected, this control action vanishes in time, as shown in Figure \ref{fig:state_ev}(b).

\begin{figure}
\includegraphics[width=\textwidth]{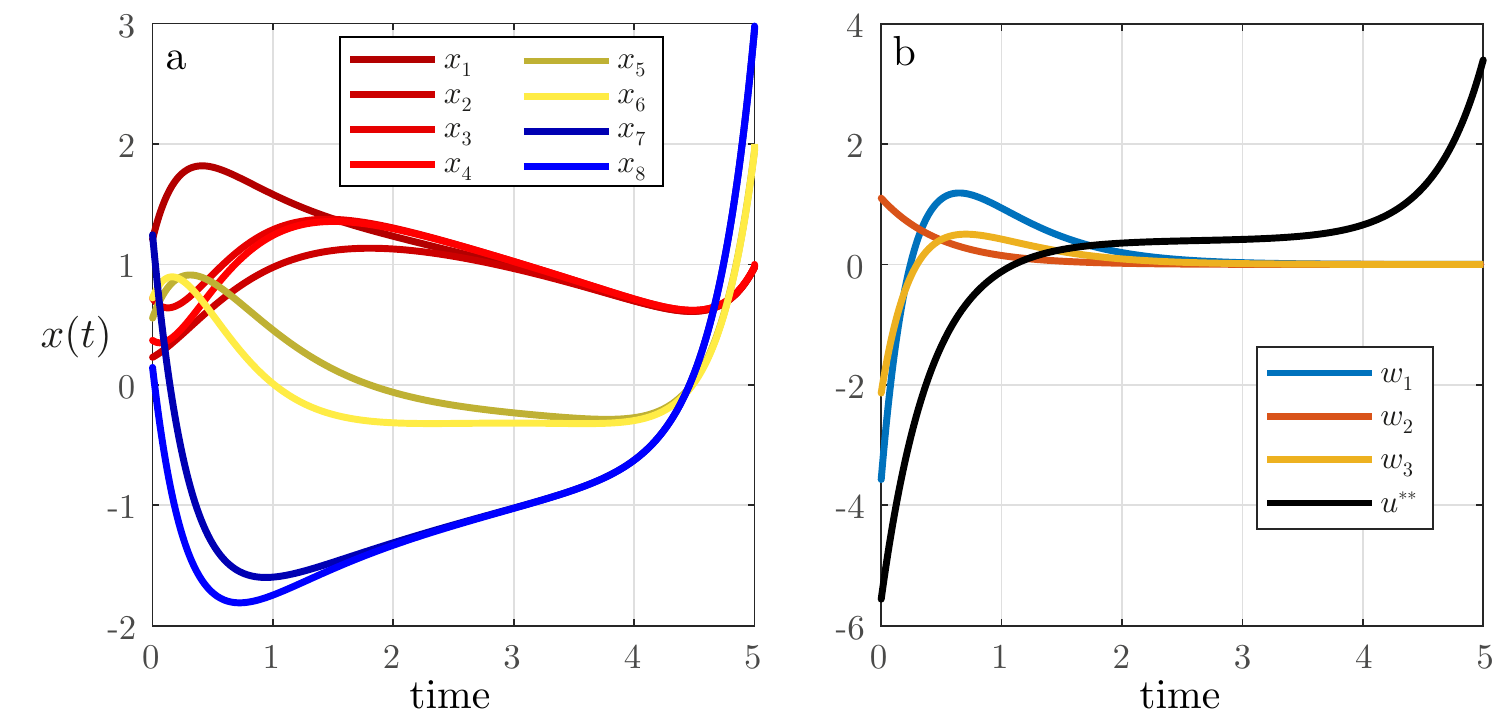}
\centering
\caption{(a) State trajectories of the original network. In red the trajectories of nodes in cluster $\C_1$ and in yellows those of nodes in cluster $\C_2$, in blue those of nodes in cluster $C_3$. (b) Control inputs.}\label{fig:state_ev}
\end{figure}

Applying Algorithm 1 to the eight node network in Fig. \ref{fig:example_net} yielded a selection of six driver nodes in order to stabilize $\mathcal{X}_{or}$, that is, $75\%$ of the network nodes. We now consider a larger network with $N=48$ nodes, shown in Figure \ref{fig:state_ev_b}a, obtained using the algorithm proposed in \citep{klickstein2018generating}. We assume that the same input signal $u$ is injected in all the nodes $i$ such that $21 \leq i \leq 35$ (the yellow nodes in the figure). A study of the symmetries of the pair $(A,B)$ for this network shows that there are $K=3$ orbital clusters with $\mathcal{C}_1:=\lbrace i:i\leq 20\rbrace$, $\mathcal{C}_2:=\lbrace i:21\leq i\leq 36\rbrace$, and $\mathcal{C}_3:=\lbrace i :i\geq 37\rbrace$ defining the quotient network in Fig. \ref{fig:state_ev_b}b. Applying the transformation in eq. \eqref{eq:transformation} and computing the eigenvalues of the matrix $A_{\perp}$ in eq. \eqref{eq:trasformed_net}, we find that $|\Lambda^{\perp}|= 8$ and that $\sum_{i:\lambda_i\in\Lambda^{\perp}} \mu_i = 19$, that is, the number of eigenvectors associated to the non-stable eigenvalues of $A_{\perp}$ is $19$. Hence, in order to ensure the network in Fig. \ref{fig:state_ev_b}a achieves group consensus we need to select an additional set of driver nodes defining the matrix $D$ in eq. \eqref{eq:contr_net_b}. To do so, we apply Algorithm 1 finding that eight input signals, i.e., a matrix $D$ with eight columns, are sufficient to stabilize the dynamics transverse to $\mathcal{X}_{or}$. Notably, only $11$ rows of the matrix $D$ encompass at least one nonzero entry, and thus only $11$ driver nodes, roughly $23\%$ of the network nodes, are sufficient to stabilize $\mathcal{X}_{or}$, five of which were already nodes in which the input signal $u$ is injected. In the appendix we give all the details on the driver node selection procedure for this example, showing that the bound in Corollary \ref{cor:min_drivers_b} is achieved also for the $48$ node network considered here. Fig. \ref{fig:state_ev_b}c, shows the trajectory generated by the joint action of an optimal controller $u^{**}$ which solves the problem 
\begin{equation}\label{eq:opt_quotient_b}
\begin{aligned}
\min_{u} &\dfrac{1}{2}\int_0^{5} u^{T}(t)u(t)dt\\
s.t.\\
\dot z^{\parallel} = &A_{\parallel}z^{\parallel}+B_{\parallel}u\\
z^{\parallel}(0) = & T^{\parallel} x_0 \\
z^{\parallel}(1) = &T^{\parallel}[\mathbf{1}_{1 \times 20} \;  \mathbf{2}_{1\times 16} \;  \mathbf{3}_{1\times 12}]^T = [\sqrt{20}\;\; 8\;\; 3\sqrt{12}]^T
\end{aligned}
\end{equation}
and of a stabilizing state feedback control action $w$ designed on the pair $(A_{\perp},D_{\perp})$ which places all the formerly unstable eigenvalues of $A_{\perp}$ in $-10$. As can be seen from the figure, group consensus is achieved starting from an initial condition that lies outside $\mathcal{X}_{or}$. Figure \ref{fig:state_ev_b}d shows the control inputs $u^{**}$ and $w_i(t)$ $i=1,\dots,8$.

\begin{figure}
\includegraphics[width=\textwidth]{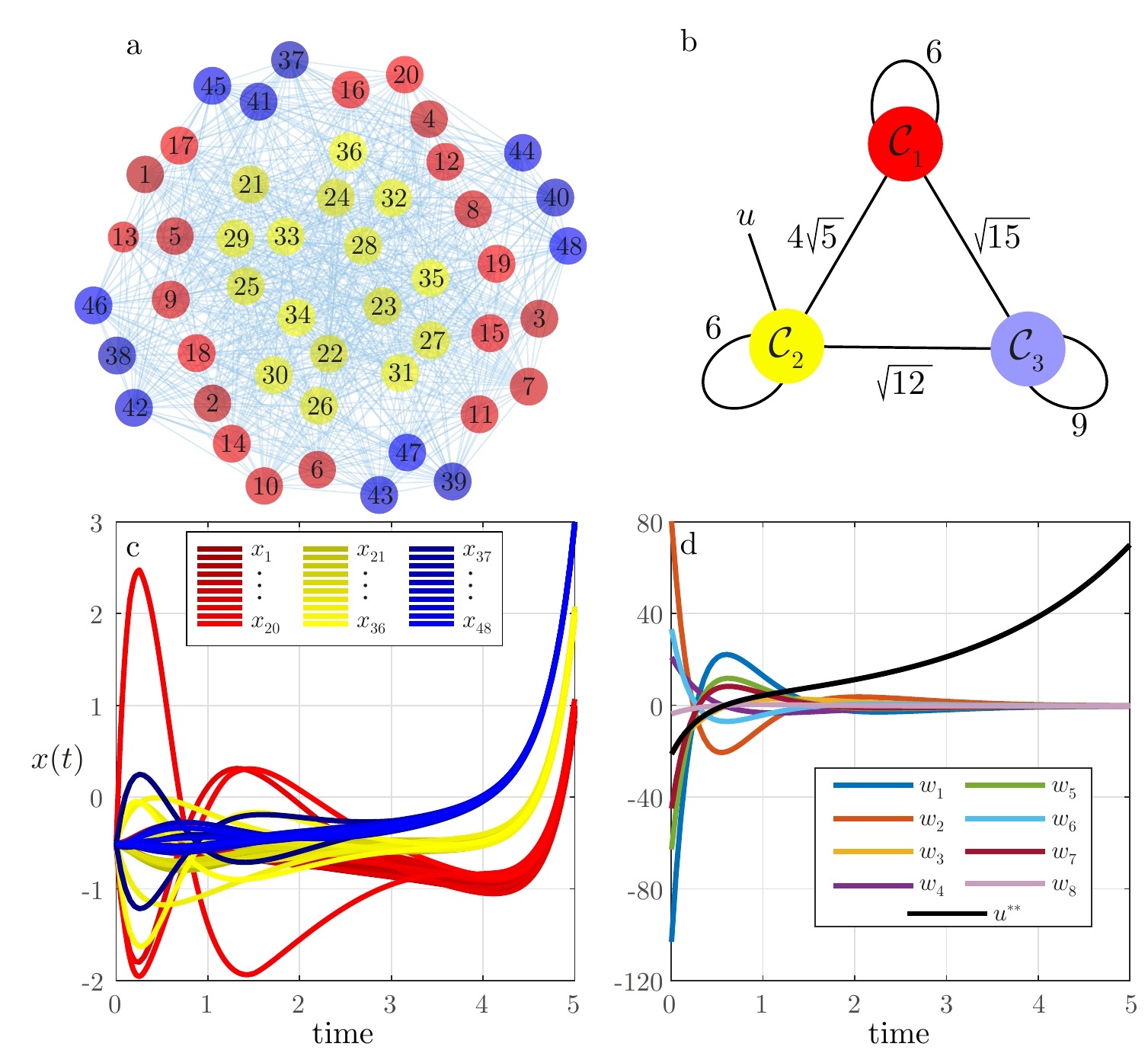}
\centering
\caption{(a) The 48 node random network with 3 orbital cluster and (b) its three node quotient network. (c) Controlled state trajectories of the network nodes driven towards the group consensus state through the joint action of the optimal control input $u$ and of the stabilizing action $w$. In red the trajectories of the nodes in cluster $\C_1$, in yellow those of the nodes in cluster $\C_2$, and in blue those of the nodes in cluster $\C_3$. (d) Time evolution of the control inputs.}\label{fig:state_ev_b}
\end{figure}

\section{Conclusions}

Motivated by the observation that symmetries induce both loss of controllability and the emergence of group consensus, in this work we studied the controllability properties of networks endowed of symmetries. We found that  controllability is lost in directions orthogonal to the group consensus subspace, but it is still possible to control the consensus state either if the network initial condition belongs to the group consensus subspace, or if the subsystem of the dynamics orthogonal to this subspace is asymptotically stable. Moreover, we showed that when the network controllable subspace coincides with the group consensus subspace, we can control consensus by designing control strategies on a lower-dimensional network, the quotient network, thus reducing the computational burden. We also considered the issue of stabilizability of the network dynamics and provided a simple algorithm to place additional control inputs that ensure that the group consensus subspace is stabilizable. By using the IRR transformation of the network symmetry group, we provided bounds on the minimum number of additional inputs and on the number of driver nodes that are needed to achieve stabilizability. We demonstrated our theoretical analysis through two representative numerical examples.

\appendix

\section{Stabilizing the cluster consesnsus on the example in Fig. \ref{fig:state_ev_b}}
The adjacency matrix of the proposed network is\\

\resizebox{\textwidth}{!}{$
\begin{array}{rc}
A=&\left[\begin{array}{cccccccccccccccccccccccccccccccccccccccccccccccc}
0&1&1&1&0&0&0&0&0&0&0&0&0&0&0&0&0&1&1&1&1&0&1&0&1&0&1&0&1&0&1&0&1&0&1&0&1&0&0&0&1&0&0&0&1&0&0&0\\
1&0&1&1&1&0&0&0&0&0&0&0&0&0&0&0&0&0&1&1&0&1&0&1&0&1&0&1&0&1&0&1&0&1&0&1&0&1&0&0&0&1&0&0&0&1&0&0\\
1&1&0&1&1&1&0&0&0&0&0&0&0&0&0&0&0&0&0&1&1&0&1&0&1&0&1&0&1&0&1&0&1&0&1&0&0&0&1&0&0&0&1&0&0&0&1&0\\
1&1&1&0&1&1&1&0&0&0&0&0&0&0&0&0&0&0&0&0&0&1&0&1&0&1&0&1&0&1&0&1&0&1&0&1&0&0&0&1&0&0&0&1&0&0&0&1\\
0&1&1&1&0&1&1&1&0&0&0&0&0&0&0&0&0&0&0&0&1&0&1&0&1&0&1&0&1&0&1&0&1&0&1&0&1&0&0&0&1&0&0&0&1&0&0&0\\
0&0&1&1&1&0&1&1&1&0&0&0&0&0&0&0&0&0&0&0&0&1&0&1&0&1&0&1&0&1&0&1&0&1&0&1&0&1&0&0&0&1&0&0&0&1&0&0\\
0&0&0&1&1&1&0&1&1&1&0&0&0&0&0&0&0&0&0&0&1&0&1&0&1&0&1&0&1&0&1&0&1&0&1&0&0&0&1&0&0&0&1&0&0&0&1&0\\
0&0&0&0&1&1&1&0&1&1&1&0&0&0&0&0&0&0&0&0&0&1&0&1&0&1&0&1&0&1&0&1&0&1&0&1&0&0&0&1&0&0&0&1&0&0&0&1\\
0&0&0&0&0&1&1&1&0&1&1&1&0&0&0&0&0&0&0&0&1&0&1&0&1&0&1&0&1&0&1&0&1&0&1&0&1&0&0&0&1&0&0&0&1&0&0&0\\
0&0&0&0&0&0&1&1&1&0&1&1&1&0&0&0&0&0&0&0&0&1&0&1&0&1&0&1&0&1&0&1&0&1&0&1&0&1&0&0&0&1&0&0&0&1&0&0\\
0&0&0&0&0&0&0&1&1&1&0&1&1&1&0&0&0&0&0&0&1&0&1&0&1&0&1&0&1&0&1&0&1&0&1&0&0&0&1&0&0&0&1&0&0&0&1&0\\
0&0&0&0&0&0&0&0&1&1&1&0&1&1&1&0&0&0&0&0&0&1&0&1&0&1&0&1&0&1&0&1&0&1&0&1&0&0&0&1&0&0&0&1&0&0&0&1\\
0&0&0&0&0&0&0&0&0&1&1&1&0&1&1&1&0&0&0&0&1&0&1&0&1&0&1&0&1&0&1&0&1&0&1&0&1&0&0&0&1&0&0&0&1&0&0&0\\
0&0&0&0&0&0&0&0&0&0&1&1&1&0&1&1&1&0&0&0&0&1&0&1&0&1&0&1&0&1&0&1&0&1&0&1&0&1&0&0&0&1&0&0&0&1&0&0\\
0&0&0&0&0&0&0&0&0&0&0&1&1&1&0&1&1&1&0&0&1&0&1&0&1&0&1&0&1&0&1&0&1&0&1&0&0&0&1&0&0&0&1&0&0&0&1&0\\
0&0&0&0&0&0&0&0&0&0&0&0&1&1&1&0&1&1&1&0&0&1&0&1&0&1&0&1&0&1&0&1&0&1&0&1&0&0&0&1&0&0&0&1&0&0&0&1\\
0&0&0&0&0&0&0&0&0&0&0&0&0&1&1&1&0&1&1&1&1&0&1&0&1&0&1&0&1&0&1&0&1&0&1&0&1&0&0&0&1&0&0&0&1&0&0&0\\
1&0&0&0&0&0&0&0&0&0&0&0&0&0&1&1&1&0&1&1&0&1&0&1&0&1&0&1&0&1&0&1&0&1&0&1&0&1&0&0&0&1&0&0&0&1&0&0\\
1&1&0&0&0&0&0&0&0&0&0&0&0&0&0&1&1&1&0&1&1&0&1&0&1&0&1&0&1&0&1&0&1&0&1&0&0&0&1&0&0&0&1&0&0&0&1&0\\
1&1&1&0&0&0&0&0&0&0&0&0&0&0&0&0&1&1&1&0&0&1&0&1&0&1&0&1&0&1&0&1&0&1&0&1&0&0&0&1&0&0&0&1&0&0&0&1\\
1&0&1&0&1&0&1&0&1&0&1&0&1&0&1&0&1&0&1&0&0&1&1&1&0&0&0&0&0&0&0&0&0&1&1&1&1&0&0&0&1&0&0&0&1&0&0&0\\
0&1&0&1&0&1&0&1&0&1&0&1&0&1&0&1&0&1&0&1&1&0&1&1&1&0&0&0&0&0&0&0&0&0&1&1&0&1&0&0&0&1&0&0&0&1&0&0\\
1&0&1&0&1&0&1&0&1&0&1&0&1&0&1&0&1&0&1&0&1&1&0&1&1&1&0&0&0&0&0&0&0&0&0&1&0&0&1&0&0&0&1&0&0&0&1&0\\
0&1&0&1&0&1&0&1&0&1&0&1&0&1&0&1&0&1&0&1&1&1&1&0&1&1&1&0&0&0&0&0&0&0&0&0&0&0&0&1&0&0&0&1&0&0&0&1\\
1&0&1&0&1&0&1&0&1&0&1&0&1&0&1&0&1&0&1&0&0&1&1&1&0&1&1&1&0&0&0&0&0&0&0&0&1&0&0&0&1&0&0&0&1&0&0&0\\
0&1&0&1&0&1&0&1&0&1&0&1&0&1&0&1&0&1&0&1&0&0&1&1&1&0&1&1&1&0&0&0&0&0&0&0&0&1&0&0&0&1&0&0&0&1&0&0\\
1&0&1&0&1&0&1&0&1&0&1&0&1&0&1&0&1&0&1&0&0&0&0&1&1&1&0&1&1&1&0&0&0&0&0&0&0&0&1&0&0&0&1&0&0&0&1&0\\
0&1&0&1&0&1&0&1&0&1&0&1&0&1&0&1&0&1&0&1&0&0&0&0&1&1&1&0&1&1&1&0&0&0&0&0&0&0&0&1&0&0&0&1&0&0&0&1\\
1&0&1&0&1&0&1&0&1&0&1&0&1&0&1&0&1&0&1&0&0&0&0&0&0&1&1&1&0&1&1&1&0&0&0&0&1&0&0&0&1&0&0&0&1&0&0&0\\
0&1&0&1&0&1&0&1&0&1&0&1&0&1&0&1&0&1&0&1&0&0&0&0&0&0&1&1&1&0&1&1&1&0&0&0&0&1&0&0&0&1&0&0&0&1&0&0\\
1&0&1&0&1&0&1&0&1&0&1&0&1&0&1&0&1&0&1&0&0&0&0&0&0&0&0&1&1&1&0&1&1&1&0&0&0&0&1&0&0&0&1&0&0&0&1&0\\
0&1&0&1&0&1&0&1&0&1&0&1&0&1&0&1&0&1&0&1&0&0&0&0&0&0&0&0&1&1&1&0&1&1&1&0&0&0&0&1&0&0&0&1&0&0&0&1\\
1&0&1&0&1&0&1&0&1&0&1&0&1&0&1&0&1&0&1&0&0&0&0&0&0&0&0&0&0&1&1&1&0&1&1&1&1&0&0&0&1&0&0&0&1&0&0&0\\
0&1&0&1&0&1&0&1&0&1&0&1&0&1&0&1&0&1&0&1&1&0&0&0&0&0&0&0&0&0&1&1&1&0&1&1&0&1&0&0&0&1&0&0&0&1&0&0\\
1&0&1&0&1&0&1&0&1&0&1&0&1&0&1&0&1&0&1&0&1&1&0&0&0&0&0&0&0&0&0&1&1&1&0&1&0&0&1&0&0&0&1&0&0&0&1&0\\
0&1&0&1&0&1&0&1&0&1&0&1&0&1&0&1&0&1&0&1&1&1&1&0&0&0&0&0&0&0&0&0&1&1&1&0&0&0&0&1&0&0&0&1&0&0&0&1\\
1&0&0&0&1&0&0&0&1&0&0&0&1&0&0&0&1&0&0&0&1&0&0&0&1&0&0&0&1&0&0&0&1&0&0&0&0&1&1&1&1&0&1&0&1&1&1&1\\
0&1&0&0&0&1&0&0&0&1&0&0&0&1&0&0&0&1&0&0&0&1&0&0&0&1&0&0&0&1&0&0&0&1&0&0&1&0&1&1&1&1&0&1&0&1&1&1\\
0&0&1&0&0&0&1&0&0&0&1&0&0&0&1&0&0&0&1&0&0&0&1&0&0&0&1&0&0&0&1&0&0&0&1&0&1&1&0&1&1&1&1&0&1&0&1&1\\
0&0&0&1&0&0&0&1&0&0&0&1&0&0&0&1&0&0&0&1&0&0&0&1&0&0&0&1&0&0&0&1&0&0&0&1&1&1&1&0&1&1&1&1&0&1&0&1\\
1&0&0&0&1&0&0&0&1&0&0&0&1&0&0&0&1&0&0&0&1&0&0&0&1&0&0&0&1&0&0&0&1&0&0&0&1&1&1&1&0&1&1&1&1&0&1&0\\
0&1&0&0&0&1&0&0&0&1&0&0&0&1&0&0&0&1&0&0&0&1&0&0&0&1&0&0&0&1&0&0&0&1&0&0&0&1&1&1&1&0&1&1&1&1&0&1\\
0&0&1&0&0&0&1&0&0&0&1&0&0&0&1&0&0&0&1&0&0&0&1&0&0&0&1&0&0&0&1&0&0&0&1&0&1&0&1&1&1&1&0&1&1&1&1&0\\
0&0&0&1&0&0&0&1&0&0&0&1&0&0&0&1&0&0&0&1&0&0&0&1&0&0&0&1&0&0&0&1&0&0&0&1&0&1&0&1&1&1&1&0&1&1&1&1\\
1&0&0&0&1&0&0&0&1&0&0&0&1&0&0&0&1&0&0&0&1&0&0&0&1&0&0&0&1&0&0&0&1&0&0&0&1&0&1&0&1&1&1&1&0&1&1&1\\
0&1&0&0&0&1&0&0&0&1&0&0&0&1&0&0&0&1&0&0&0&1&0&0&0&1&0&0&0&1&0&0&0&1&0&0&1&1&0&1&0&1&1&1&1&0&1&1\\
0&0&1&0&0&0&1&0&0&0&1&0&0&0&1&0&0&0&1&0&0&0&1&0&0&0&1&0&0&0&1&0&0&0&1&0&1&1&1&0&1&0&1&1&1&1&0&1\\
0&0&0&1&0&0&0&1&0&0&0&1&0&0&0&1&0&0&0&1&0&0&0&1&0&0&0&1&0&0&0&1&0&0&0&1&1&1&1&1&0&1&0&1&1&1&1&0
\end{array}\right],\\
B^T=&\left[\begin{array}{cccccccccccccccccccccccccccccccccccccccccccccccc}
0&0&0&0&0&0&0&0&0&0&0&0&0&0&0&0&0&0&0&0&1&1&1&1&1&1&1&1&1&1&1&1&1&1&1&1&0&0&0&0&0&0&0&0&0&0&0&0
\end{array}
\right].
\end{array}
$}

\newpage
Using the algorithm in \citep{peso14}, we compute the transformation to the IRR coordinate system of this network, that is\\[2mm]
\resizebox{\textwidth}{!}{
\setlength\arraycolsep{1pt}$	
100\ T_{or}=\left[\begin{array}{cccccccccccccccccccccccccccccccccccccccccccccccc}
22&22&22&22&22&22&22&22&22&22&22&22&22&22&22&22&22&22&22&22&0&0&0&0&0&0&0&0&0&0&0&0&0&0&0&0&0&0&0&0&0&0&0&0&0&0&0&0
\\0&0&0&0&0&0&0&0&0&0&0&0&0&0&0&0&0&0&0&0&25&25&25&25&25&25&25&25&25&25&25&25&25&25&25&25&0&0&0&0&0&0&0&0&0&0&0&0
\\0&0&0&0&0&0&0&0&0&0&0&0&0&0&0&0&0&0&0&0&0&0&0&0&0&0&0&0&0&0&0&0&0&0&0&0&29&29&29&29&29&29&29&29&29&29&29&29
\\22&-22&22&-22&22&-22&22&-22&22&-22&22&-22&22&-22&22&-22&22&-22&22&-22&0&0&0&0&0&0&0&0&0&0&0&0&0&0&0&0&0&0&0&0&0&0&0&0&0&0&0&0
\\0&0&0&0&0&0&0&0&0&0&0&0&0&0&0&0&0&0&0&0&-25&25&-25&25&-25&25&-25&25&-25&25&-25&25&-25&25&-25&25&0&0&0&0&0&0&0&0&0&0&0&0
\\0&0&0&0&0&0&0&0&0&0&0&0&0&0&0&0&0&0&0&0&0&0&0&0&0&0&0&0&0&0&0&0&0&0&0&0&-29&29&-29&29&-29&29&-29&29&-29&29&-29&29
\\32&0&-32&0&32&0&-32&0&32&0&-32&0&32&0&-32&0&32&0&-32&0&0&0&0&0&0&0&0&0&0&0&0&0&0&0&0&0&0&0&0&0&0&0&0&0&0&0&0&0
\\0&0&0&0&0&0&0&0&0&0&0&0&0&0&0&0&0&0&0&0&35&0&-35&0&35&0&-35&0&35&0&-35&0&35&0&-35&0&0&0&0&0&0&0&0&0&0&0&0&0
\\0&0&0&0&0&0&0&0&0&0&0&0&0&0&0&0&0&0&0&0&0&0&0&0&0&0&0&0&0&0&0&0&0&0&0&0&-41&0&41&0&-41&0&41&0&-41&0&41&0
\\0&-32&0&32&0&-32&0&32&0&-32&0&32&0&-32&0&32&0&-32&0&32&0&0&0&0&0&0&0&0&0&0&0&0&0&0&0&0&0&0&0&0&0&0&0&0&0&0&0&0
\\0&0&0&0&0&0&0&0&0&0&0&0&0&0&0&0&0&0&0&0&0&-35&0&35&0&-35&0&35&0&-35&0&35&0&-35&0&35&0&0&0&0&0&0&0&0&0&0&0&0
\\0&0&0&0&0&0&0&0&0&0&0&0&0&0&0&0&0&0&0&0&0&0&0&0&0&0&0&0&0&0&0&0&0&0&0&0&0&41&0&-41&0&41&0&-41&0&41&0&-41
\\7&-27&-24&13&31&7&-27&-24&13&31&7&-27&-24&13&31&7&-27&-24&13&31&0&0&0&0&0&0&0&0&0&0&0&0&0&0&0&0&0&0&0&0&0&0&0&0&0&0&0&0
\\31&16&-21&-29&3&31&16&-21&-29&3&31&16&-21&-29&3&31&16&-21&-29&3&0&0&0&0&0&0&0&0&0&0&0&0&0&0&0&0&0&0&0&0&0&0&0&0&0&0&0&0
\\-29&24&-17&8&2&-12&20&-27&31&-32&29&-24&17&-8&-2&12&-20&27&-31&32&0&0&0&0&0&0&0&0&0&0&0&0&0&0&0&0&0&0&0&0&0&0&0&0&0&0&0&0
\\-12&20&-27&31&-32&29&-24&17&-8&-2&12&-20&27&-31&32&-29&24&-17&8&2&0&0&0&0&0&0&0&0&0&0&0&0&0&0&0&0&0&0&0&0&0&0&0&0&0&0&0&0
\\5&28&28&5&-23&-31&-14&15&31&22&-5&-28&-28&-5&23&31&14&-15&-31&-22&0&0&0&0&0&0&0&0&0&0&0&0&0&0&0&0&0&0&0&0&0&0&0&0&0&0&0&0
\\-31&-14&15&31&22&-5&-28&-28&-5&23&31&14&-15&-31&-22&5&28&28&5&-23&0&0&0&0&0&0&0&0&0&0&0&0&0&0&0&0&0&0&0&0&0&0&0&0&0&0&0&0
\\-5&-28&22&14&-31&5&28&-22&-14&31&-5&-28&22&14&-31&5&28&-22&-14&31&0&0&0&0&0&0&0&0&0&0&0&0&0&0&0&0&0&0&0&0&0&0&0&0&0&0&0&0
\\31&-14&-22&28&5&-31&14&22&-28&-5&31&-14&-22&28&5&-31&14&22&-28&-5&0&0&0&0&0&0&0&0&0&0&0&0&0&0&0&0&0&0&0&0&0&0&0&0&0&0&0&0
\\1&-19&30&-30&18&1&-19&30&-30&18&1&-19&30&-30&18&1&-19&30&-30&18&0&0&0&0&0&0&0&0&0&0&0&0&0&0&0&0&0&0&0&0&0&0&0&0&0&0&0&0
\\-32&25&-9&-10&26&-32&25&-9&-10&26&-32&25&-9&-10&26&-32&25&-9&-10&26&0&0&0&0&0&0&0&0&0&0&0&0&0&0&0&0&0&0&0&0&0&0&0&0&0&0&0&0
\\29&-27&3&24&-31&12&16&-31&21&7&-29&27&-3&-24&31&-12&-16&31&-21&-7&0&0&0&0&0&0&0&0&0&0&0&0&0&0&0&0&0&0&0&0&0&0&0&0&0&0&0&0
\\-12&-16&31&-21&-7&29&-27&3&24&-31&12&16&-31&21&7&-29&27&-3&-24&31&0&0&0&0&0&0&0&0&0&0&0&0&0&0&0&0&0&0&0&0&0&0&0&0&0&0&0&0
\\-28&-31&-22&-5&15&28&31&22&5&-15&-28&-31&-22&-5&15&28&31&22&5&-15&0&0&0&0&0&0&0&0&0&0&0&0&0&0&0&0&0&0&0&0&0&0&0&0&0&0&0&0
\\-14&5&23&31&28&14&-5&-23&-31&-28&-14&5&23&31&28&14&-5&-23&-31&-28&0&0&0&0&0&0&0&0&0&0&0&0&0&0&0&0&0&0&0&0&0&0&0&0&0&0&0&0
\\32&30&25&17&8&-2&-11&-20&-26&-31&-32&-30&-25&-17&-8&2&11&20&26&31&0&0&0&0&0&0&0&0&0&0&0&0&0&0&0&0&0&0&0&0&0&0&0&0&0&0&0&0
\\-2&-11&-20&-26&-31&-32&-30&-25&-17&-8&2&11&20&26&31&32&30&25&17&8&0&0&0&0&0&0&0&0&0&0&0&0&0&0&0&0&0&0&0&0&0&0&0&0&0&0&0&0
\\0&0&0&0&0&0&0&0&0&0&0&0&0&0&0&0&0&0&0&0&-29&7&35&20&-20&-35&-7&29&29&-7&-35&-20&20&35&7&-29&0&0&0&0&0&0&0&0&0&0&0&0
\\0&0&0&0&0&0&0&0&0&0&0&0&0&0&0&0&0&0&0&0&-20&-35&-7&29&29&-7&-35&-20&20&35&7&-29&-29&7&35&20&0&0&0&0&0&0&0&0&0&0&0&0
\\0&0&0&0&0&0&0&0&0&0&0&0&0&0&0&0&0&0&0&0&35&-29&19&-7&-7&20&-30&35&-35&29&-19&7&7&-20&30&-35&0&0&0&0&0&0&0&0&0&0&0&0
\\0&0&0&0&0&0&0&0&0&0&0&0&0&0&0&0&0&0&0&0&-7&20&-30&35&-35&29&-19&7&7&-20&30&-35&35&-29&19&-7&0&0&0&0&0&0&0&0&0&0&0&0
\\0&0&0&0&0&0&0&0&0&0&0&0&0&0&0&0&0&0&0&0&-9&-28&31&5&-34&22&18&-35&9&28&-31&-5&34&-22&-18&35&0&0&0&0&0&0&0&0&0&0&0&0
\\0&0&0&0&0&0&0&0&0&0&0&0&0&0&0&0&0&0&0&0&34&-22&-18&35&-9&-28&31&5&-34&22&18&-35&9&28&-31&-5&0&0&0&0&0&0&0&0&0&0&0&0
\\0&0&0&0&0&0&0&0&0&0&0&0&0&0&0&0&0&0&0&0&-22&-10&3&17&27&34&35&31&22&10&-3&-17&-27&-34&-35&-31&0&0&0&0&0&0&0&0&0&0&0&0
\\0&0&0&0&0&0&0&0&0&0&0&0&0&0&0&0&0&0&0&0&27&34&35&31&22&10&-3&-17&-27&-34&-35&-31&-22&-10&3&17&0&0&0&0&0&0&0&0&0&0&0&0
\\0&0&0&0&0&0&0&0&0&0&0&0&0&0&0&0&0&0&0&0&-30&11&-18&34&30&-11&18&-34&-30&11&-18&34&30&-11&18&-34&0&0&0&0&0&0&0&0&0&0&0&0
\\0&0&0&0&0&0&0&0&0&0&0&0&0&0&0&0&0&0&0&0&30&-11&-37&11&-30&11&37&-11&30&-11&-37&11&-30&11&37&-11&0&0&0&0&0&0&0&0&0&0&0&0
\\0&0&0&0&0&0&0&0&0&0&0&0&0&0&0&0&0&0&0&0&-16&16&-28&-35&16&-16&28&35&-16&16&-28&-35&16&-16&28&35&0&0&0&0&0&0&0&0&0&0&0&0
\\0&0&0&0&0&0&0&0&0&0&0&0&0&0&0&0&0&0&0&0&21&45&5&6&-21&-45&-5&-6&21&45&5&6&-21&-45&-5&-6&0&0&0&0&0&0&0&0&0&0&0&0
\\0&0&0&0&0&0&0&0&0&0&0&0&0&0&0&0&0&0&0&0&0&0&0&0&0&0&0&0&0&0&0&0&0&0&0&0&35&-35&0&35&-35&0&35&-35&0&35&-35&0
\\0&0&0&0&0&0&0&0&0&0&0&0&0&0&0&0&0&0&0&0&0&0&0&0&0&0&0&0&0&0&0&0&0&0&0&0&-20&-20&41&-20&-20&41&-20&-20&41&-20&-20&41
\\0&0&0&0&0&0&0&0&0&0&0&0&0&0&0&0&0&0&0&0&0&0&0&0&0&0&0&0&0&0&0&0&0&0&0&0&-35&20&0&-20&35&-41&35&-20&0&20&-35&41
\\0&0&0&0&0&0&0&0&0&0&0&0&0&0&0&0&0&0&0&0&0&0&0&0&0&0&0&0&0&0&0&0&0&0&0&0&-20&35&-41&35&-20&0&20&-35&41&-35&20&0
\\0&0&0&0&0&0&0&0&0&0&0&0&0&0&0&0&0&0&0&0&0&0&0&0&0&0&0&0&0&0&0&0&0&0&0&0&-38&-32&6&38&32&-6&-38&-32&6&38&32&-6
\\0&0&0&0&0&0&0&0&0&0&0&0&0&0&0&0&0&0&0&0&0&0&0&0&0&0&0&0&0&0&0&0&0&0&0&0&-15&26&40&15&-26&-40&-15&26&40&15&-26&-40
\\0&0&0&0&0&0&0&0&0&0&0&0&0&0&0&0&0&0&0&0&0&0&0&0&0&0&0&0&0&0&0&0&0&0&0&0&-35&-20&0&20&35&41&35&20&0&-20&-35&-41
\\0&0&0&0&0&0&0&0&0&0&0&0&0&0&0&0&0&0&0&0&0&0&0&0&0&0&0&0&0&0&0&0&0&0&0&0&20&35&41&35&20&0&-20&-35&-41&-35&-20&0
\end{array}\right].
$}
\medskip

Note that each row of the transformation is cluster specific, that is, each row has non-zero entries in the elements corresponding to only one of the clusters. Applying this transformation to our example, we obtain\\[2mm]
\resizebox{\textwidth}{!}{
\setlength\arraycolsep{5pt}$
\tilde{A}=\left[\begin{array}{ccccc}
A_\parallel & 0 & 0 & 0   & 0 \\
0 & A_{\perp}^{1} & 0 & 0 & 0 \\
0 & 0 & A_{\perp}^{2} & 0 & 0 \\
0 & 0 & 0 & A_{\perp}^{3} & 0 \\
0 & 0 & 0 & 0 & A_{\perp}^{4,\ldots,\textcolor{blue}{39}}
\end{array}\right],
\scriptsize
\quad
\begin{array}{ll}
A_\parallel=
\begin{bmatrix}
6 & 4\sqrt{5}&\sqrt{15}\\
4\sqrt{5}& 6&\sqrt{12}\\
\sqrt{15}&\sqrt{12}&9
\end{bmatrix},
\quad
A_{\perp}^{1}=
\begin{bmatrix}
-2&-9&-3.9\\
-9&-2&3.5\\
-3.9&3.5&1
\end{bmatrix},\quad
A_{\perp}^{2}=A_{\perp}^{3}=
\begin{bmatrix}
-2&0&-3.9\\
0&-2&-3.5\\
-3.9&-3.5&-1
\end{bmatrix}\\[4mm]
A_{\perp}^{4,\ldots,19} = 
\text{diag}(-2.6,-2.6,-1.5,-1.5,-1.3,-1.3,-0.6,-0.6,-0.4,-0.4,0.1,0.1,1.6,1.6,4.7,4.7)\\[2mm]
A_{\perp}^{20,\ldots,31} =\text{diag}(-2.5,-2.5,-1.2,-1.2,-0.3,-0.3,4,4,0,0,0,0)\\[2mm]
A_{\perp}^{32,\ldots,39}
=\text{diag}(0,0,-2.7,-2.7,-2,-2,0.7,0.7)
\end{array}$}
\smallskip

\noindent
where we have highlighted the block structure of the matrix $\tilde A$. Note that the first three rows of $T_{or}$ span the cluster consensus subspace $\mathcal{X}_{or}$. Then we have three sets of three rows of the so called \emph{intertwined} symmetry-breaks \citep{peso14}, that define three $3\times 3$ blocks $A_\perp^1,\ldots,A_\perp^3$ of $\tilde A$ each governing the dynamics along an $A$-invariant subspace. Any one of these blocks is generated by three rows of the matrix $T_{or}$ each specific of a different cluster. The eigenvectors of the matrix $A$ generating these three dimensional invariant subspaces have therefore non-zero entries in all their elements (since they involve all the three clusters/all the nodes of the network). The remaining 36 rows of $T_{or}$ define 36 monodimensional blocks of $\tilde{A}$ ($A_\perp^4,\ldots,A_\perp^{39}$), and are therefore themselves eigenvectors of the matrix $A$. The first 16 are specific of cluster $\C_1$, the next 12 are specific of cluster $\C_2$, and finally the last 8 are specific of cluster $\C_3$.

The transverse non-stable eigenvalues that define $\Lambda^\perp$ are the 16 non-negative monodimensional block of $\tilde A$, together with three other positive eigenvalues, one for each fo the 3x3 blocks of $A_\perp$. As a result \\[2mm]
\resizebox{\textwidth}{!}{
\begin{tikzpicture}
\setlength\arraycolsep{5pt}
\node at (-8.3,12){
$
\Lambda^\perp=\Big[
$};
\node at (8.5,12){$\Big]$};
\node at (-7.2,12){$9.9$};
\node at (-6,12){$3.7$};
\draw [decorate,decoration={brace,amplitude=5pt}] (-6.7,11.5) -- (-5.3,11.5) ;
\node at (-4.2,12){$0.1$};
\draw [decorate,decoration={brace,amplitude=5pt}] (-4.9,11.5) -- (-3.5,11.5) ;
\node at (-2.5,12){$1.6$};
\draw [decorate,decoration={brace,amplitude=5pt}] (-3.2,11.5) -- (-1.8,11.5) ;
\node at (-0.9,12){$4.7$};
\draw [decorate,decoration={brace,amplitude=5pt}] (-1.6,11.5) -- (-0.2,11.5) ;
\node at (0.8,12){$4$};
\draw [decorate,decoration={brace,amplitude=5pt}] (0.1,11.5) -- (1.5,11.5) ;
\node at (4.1,12){$0$};
\draw [decorate,decoration={brace,amplitude=5pt}] (1.8,11.5) -- (6.5,11.5) ;
\node at (7.5,12){$0.7$};
\draw [decorate,decoration={brace,amplitude=5pt}] (6.8,11.5) -- (8.2,11.5) ;
\node at (0,0.2) {
$
\phantom{\Omega_i=}\left[\begin{array}{ccccccccccccccccccc}
    1&  0&  2&  3& -1& -1&  3& -1&  3&  0&  0&  0&  0&  0&  0&  0&  0&  0&  0
\\ -1& -2&  0&  2& -2&  1&  3& -1& -2&  0&  0&  0&  0&  0&  0&  0&  0&  0&  0
\\  1&  0& -2&  1& -3&  2&  2&  3&  0&  0&  0&  0&  0&  0&  0&  0&  0&  0&  0
\\ -1&  2&  0&  0& -3&  3&  0& -2&  3&  0&  0&  0&  0&  0&  0&  0&  0&  0&  0
\\  1&  0&  2& -1& -3&  3& -2&  0& -3&  0&  0&  0&  0&  0&  0&  0&  0&  0&  0
\\ -1& -2&  0& -1& -3&  1& -3&  3&  1&  0&  0&  0&  0&  0&  0&  0&  0&  0&  0
\\  1&  0& -2& -2& -3& -1& -3& -3&  2&  0&  0&  0&  0&  0&  0&  0&  0&  0&  0
\\ -1&  2&  0& -3& -2& -2& -2&  0& -3&  0&  0&  0&  0&  0&  0&  0&  0&  0&  0
\\  1&  0&  2& -3& -1& -3&  0&  2&  2&  0&  0&  0&  0&  0&  0&  0&  0&  0&  0
\\ -1& -2&  0& -3&  0& -3&  2& -3&  1&  0&  0&  0&  0&  0&  0&  0&  0&  0&  0
\\  1&  0& -2& -3&  1& -1&  3&  1& -3&  0&  0&  0&  0&  0&  0&  0&  0&  0&  0
\\ -1&  2&  0& -2&  2&  1&  3&  1&  2&  0&  0&  0&  0&  0&  0&  0&  0&  0&  0
\\  1&  0&  2& -1&  3&  2&  2& -3&  0&  0&  0&  0&  0&  0&  0&  0&  0&  0&  0
\\ -1& -2&  0&  0&  3&  3&  0&  2& -3&  0&  0&  0&  0&  0&  0&  0&  0&  0&  0
\\  1&  0& -2&  1&  3&  3& -2&  0&  3&  0&  0&  0&  0&  0&  0&  0&  0&  0&  0
\\ -1&  2&  0&  1&  3&  1& -3& -3& -1&  0&  0&  0&  0&  0&  0&  0&  0&  0&  0
\\  1&  0&  2&  2&  3& -1& -3&  3& -2&  0&  0&  0&  0&  0&  0&  0&  0&  0&  0
\\ -1& -2&  0&  3&  2& -2& -2&  0&  3&  0&  0&  0&  0&  0&  0&  0&  0&  0&  0
\\  1&  0& -2&  3&  1& -3&  0& -2& -2&  0&  0&  0&  0&  0&  0&  0&  0&  0&  0
\\ -1&  2&  0&  3&  0& -3&  2&  3& -1&  0&  0&  0&  0&  0&  0&  0&  0&  0&  0
\\[-2mm]
\\  2&  0&  2&  0&  0&  0&  0&  0&  0&  1& -3& -3&  3& -2&  2&	0&  0&  0&  0
\\ -2& -2&  0&  0&  0&  0&  0&  0&  0&  3& -2&  1& -1&  2&  4&  0&  0&  0&  0
\\  2&  0& -2&  0&  0&  0&  0&  0&  0&  3& -1& -2& -4& -3&  1&  0&  0&  0&  0
\\ -2&  2&  0&  0&  0&  0&  0&  0&  0&  4&  0&  3&  1& -3&  1&  0&  0&  0&  0
\\  2&  0&  2&  0&  0&  0&  0&  0&  0&  3&  1&  3& -3&  2& -2&  0&  0&  0&  0
\\ -2& -2&  0&  0&  0&  0&  0&  0&  0&  2&  3& -1&  1& -2& -4&  0&  0&  0&  0
\\  2&  0& -2&  0&  0&  0&  0&  0&  0&  1&  3&  2&  4&  3& -1&  0&  0&  0&  0
\\ -2&  2&  0&  0&  0&  0&  0&  0&  0&  0&  4& -3& -1&  3& -1&  0&  0&  0&  0
\\  2&  0&  2&  0&  0&  0&  0&  0&  0& -1&  3& -3&  3& -2&  2&  0&  0&  0&  0
\\ -2& -2&  0&  0&  0&  0&  0&  0&  0& -3&  2&  1& -1&  2&  4&  0&  0&  0&  0
\\  2&  0& -2&  0&  0&  0&  0&  0&  0& -3&  1& -2& -4& -3&  1&  0&  0&  0&  0
\\ -2&  2&  0&  0&  0&  0&  0&  0&  0& -4&  0&  3&  1& -3&  1&  0&  0&  0&  0
\\  2&  0&  2&  0&  0&  0&  0&  0&  0& -3& -1&  3& -3&  2& -2&  0&  0&  0&  0
\\ -2& -2&  0&  0&  0&  0&  0&  0&  0& -2& -3& -1&  1& -2& -4&  0&  0&  0&  0
\\  2&  0& -2&  0&  0&  0&  0&  0&  0& -1& -3&  2&  4&  3& -1&  0&  0&  0&  0
\\ -2&  2&  0&  0&  0&  0&  0&  0&  0&  0& -4& -3& -1&  3& -1&  0&  0&  0&  0
\\[-2mm]
\\  1&  0&  3&  0&  0&  0&  0&  0&  0&  0&  0&  0&  0&  0&  0&  4& -2& -3&  3
\\ -1& -3&  0&  0&  0&  0&  0&  0&  0&  0&  0&  0&  0&  0&  0& -4& -2& -1&  4
\\  1&  0& -3&  0&  0&  0&  0&  0&  0&  0&  0&  0&  0&  0&  0&  0&  4&  1&  4
\\ -1&  3&  0&  0&  0&  0&  0&  0&  0&  0&  0&  0&  0&  0&  0&  4& -2&  3&  3
\\  1&  0&  3&  0&  0&  0&  0&  0&  0&  0&  0&  0&  0&  0&  0& -4& -2&  4&  1
\\ -1& -3&  0&  0&  0&  0&  0&  0&  0&  0&  0&  0&  0&  0&  0&  0&  4&  4& -1
\\  1&  0& -3&  0&  0&  0&  0&  0&  0&  0&  0&  0&  0&  0&  0&  4& -2&  3& -3
\\ -1&  3&  0&  0&  0&  0&  0&  0&  0&  0&  0&  0&  0&  0&  0& -4& -2&  1& -4
\\  1&  0&  3&  0&  0&  0&  0&  0&  0&  0&  0&  0&  0&  0&  0&  0&  4& -1& -4
\\ -1& -3&  0&  0&  0&  0&  0&  0&  0&  0&  0&  0&  0&  0&  0&  4& -2& -3& -3
\\  1&  0& -3&  0&  0&  0&  0&  0&  0&  0&  0&  0&  0&  0&  0& -4& -2& -4& -1
\\ -1&  3&  0&  0&  0&  0&  0&  0&  0&  0&  0&  0&  0&  0&  0&  0&  4& -4&  1
\end{array}\right]
.
$};
\end{tikzpicture}
}
where the brackets associate each $\lambda_i\in \Lambda^{\perp}$ to the eigenvectors obtained according to Remark \ref{rem:c_spec_eigs} and spanning $\Omega_i$.

We are now ready to apply Algorithm \ref{algo:main_algo} to find the driver nodes needed to stabilize $\mathcal{X}_{or}$. 
\begin{itemize}
    \item when $i=1$ we consider the eigenspace of the eigenvalue 9.9. Its dimension is 1, so we need at least one control input and two driver nodes to stabilize it. We select nodes 1 and 2 as drivers and thus $D_{1,1}=1,\ D_{1,2}=-1,\, D_{1,j}=0,\ j=3,\ldots,48$.
    \item when $i=2$ we consider the eigenspace of the eigenvalue 3.7. This eigenspace is two-dimensional, so we need another control input and another driver node to stabilize it. We then add a second (independent) column to the matrix $D$ with $D_{2,1}=1,\ D_{2,4}=-1,\ D_{2,j}=0,\ j=2,3,5,\ldots,48$.
    We then verify that $D$ has now two columns with non-zero and non-parallel projection on the eigenspace associated to the eigenvalue 3.7 by computing the elements
    \[
    \begin{array}{rclcrcl}
    \mathcal{D}^2_{1,1}&=&[1,-1,0][0,-2, 2]^T=2,&\quad&
    \mathcal{D}^2_{1,2}&=&[1,-1,0][2, 0, 0]^T=2,\\
    \mathcal{D}^2_{2,1}&=&[1,0,-1][0,-2, 2]^T=-2,&\quad&
    \mathcal{D}^2_{2,2}&=&[1,0,-1][2, 0, 0]^T=2,
    \end{array}
    \]
    of the matrix $\mathcal{D}^2$ and then verifying that this matrix is full rank as $\det(\mathcal{D}^2)=8\neq 0$.
    \item when $i=3$ we consider the eigenspace associated to the eigenvalue 0.1. Its dimension is 2, and the vectors in $D$ have a two dimensional projection on it as the elements
    \[
    \begin{array}{rclcrcl}
    \mathcal{D}^3_{1,1}&=&[1,-1,0][ 3, 2, 0]^T=1,&\quad&
    \mathcal{D}^3_{1,2}&=&[1,-1,0][-1,-2,-3]^T=1,\\
    \mathcal{D}^3_{2,1}&=&[1,0,-1][ 3, 2, 0]^T=3,&\quad&
    \mathcal{D}^3_{2,2}&=&[1,0,-1][-1,-2,-3]^T=2,
    \end{array}
    \]
    define the matrix $\mathcal{D}^3$ that is is full rank as $\det(\mathcal{D}^3)=-1\neq 0$.  
    \item when $i=4$ we consider the eigenspace associated to the eigenvalue 1.1. It's dimension is 2, and the vectors in $D$ have a two dimensional projection on it as the elements
    \[
    \begin{array}{rclcrcl}
    \mathcal{D}^4_{1,1}&=&[1,-1,0][-1, 1, 3]^T=-2,&\quad&
    \mathcal{D}^4_{1,2}&=&[1,-1,0][ 3, 3, 0]^T=0,\\
    \mathcal{D}^4_{2,1}&=&[1,0,-1][-1, 1, 3]^T=-4,&\quad&
    \mathcal{D}^4_{2,2}&=&[1,0,-1][ 3, 3, 0]^T=3,
    \end{array}
    \]
    define the matrix $\mathcal{D}^4$ that is is full rank as $\det(\mathcal{D}^4)=-6\neq 0$.
    \item when $i=5$ we consider the eigenspace associated to the eigenvalue 4.7. It's dimension is 2, and the vectors in $D$ have a two dimensional projection on it as the elements
    \[
    \begin{array}{rclcrcl}
    \mathcal{D}^5_{1,1}&=&[1,-1,0][-1,-1,-2]^T=0,&\quad&
    \mathcal{D}^5_{1,2}&=&[1,-1,0][3, -2, 3]^T=5,\\
    \mathcal{D}^5_{2,1}&=&[1,0,-1][-1,-1,-2]^T=1,&\quad&
    \mathcal{D}^5_{2,2}&=&[1,0,-1][ 3,-2, 3]^T=0,
    \end{array}
    \]
    define the matrix $\mathcal{D}^5$ that is is full rank being $\det(\mathcal{D}^5)=-5\neq 0$.
    \item the eigenspaces $\Omega_i$ when $i\geq 6$ have 0 components on cluster $\mathcal{C}_1$. As a consequence, we need to select additional drivers from the other clusters in order to stabilize them. In particular, for $i=6$ we consider the eigenspace associated to the eigenvalue 4. It's dimension is 2, and so we need at least 3 driver nodes in the cluster $\mathcal{C}_2$ in order to have a two dimensional projection on it. We then select $D_{3,21}=1,\ D_{3,22}=-1,\ D_{3,j}=0\  j=1,\ldots,20,23,\ldots,48$ and $D_{4,21}=1,\ D_{4,23}=-1,\  D_{4,j}=0,\ j=1,\ldots,20,22,24,\ldots,48$. This achieves our goal as the elements
    \[
    \begin{array}{rclcrcl}
    \mathcal{D}^6_{1,1}&=&[1,-1,0][ 1, 3, 3]^T=-2,&\quad&
    \mathcal{D}^6_{1,2}&=&[1,-1,0][-3,-2,-1]^T=1,\\
    \mathcal{D}^6_{2,1}&=&[1,0,-1][ 1, 3, 3]^T=-2,&\quad&
    \mathcal{D}^6_{2,2}&=&[1,0,-1][-3,-2,-1]^T=2,
    \end{array}
    \]
    define the matrix $\mathcal{D}^6$ that is is full rank being $\det(\mathcal{D}^6)=-2\neq 0$.
    \item when $i=7$ we consider the eigenspace associated to the eigenvalue 0. It's dimension is 6, but we can treat separately the first 4 eigenvectors, associated to cluster $\mathcal{C}_2$ and thus spanning $\Omega_7^2$, from the other 2 eigenvectors, associated to cluster $\mathcal{C}_3$ and thus spanning $\Omega_7^2$. As $|\Omega_7^2|=4$, we need to select two additional driver nodes for the matrix $D$ to have four columns with nonzero and non-parallel projection on it. We therefore select nodes 24 and 25 as drivers by adding to $D$ the columns $D_{5,21}=1,\ D_{5,24}=-1,\ D_{5,j}=0\  j=1,\ldots,20,22,23,25,\ldots,48$ and $D_{6,21}=1,\ D_{6,25}=-1, D_{6,j}=0\ j=1,\ldots,20,22,23,24,26,\ldots,48$. As the matrix
    \[
    \mathcal{D}^7=
    \begin{bmatrix}
    1&-1&0&0&0\\
    1&0&-1&0&0\\
    1&0&0&-1&0\\
    1&0&0&0&-1
    \end{bmatrix}
    \begin{bmatrix}
    -3&3 &-2& 2\\
     1&-1& 2& 4\\
    -2&-4&-3& 1\\
     3& 1&-3& 1\\
     3&-3& 2&-2
    \end{bmatrix} =
    \begin{bmatrix}
    -4& 4&-4&-2\\
    -1& 7& 1& 1\\
    -6& 2& 1& 1\\
    -6& 6&-4& 4
    \end{bmatrix}
    \]
    is full rank, then the matrix $D$ has now four columns with nonzero and non parallel projection on $|\Omega_7^2|=4$. \\
    Then, we turn our attention to $\Omega_7^3$ noting that $|\Omega_7^3|=2$. We therefore need to select three additional drivers defined by $D_{7,37}=1,\ D_{3,38}=-1,\  D_{7,j}=0,\ j=1,\ldots,36,29,\ldots,48$ and $D_{8,37}=1,\ D_{8,39}=-1, D_{8,j}=0\ j=1,\ldots,36,38,40,\ldots,48$. As the matrix
    \[
    \mathcal{D'}^7=
    \begin{bmatrix}
    1&-1&0\\
    1&0&-1
    \end{bmatrix}
    \begin{bmatrix}
     4&-2\\
    -4&-2
     0& 4
    \end{bmatrix} =
    \begin{bmatrix}
     8& 0 \\ 4&-6
    \end{bmatrix}
    \]
    is full rank being $\det(\mathcal{D'}^7)=-48\neq 0$, then the matrix $D$ has now two columns with nonzero and non-parallel projection on $\Omega_7^3$.
    \item Our procedure ends with 
    iteration $i=8$ in which we consider the eigenspace associated to the eigenvalue 0.7. Note that the matrix $D$ already has two columns with nonzero and non-parallel projection on it, namely $D_7$ and $D_8$ as the matrix
    \[
    \det(\mathcal{D}^8)=
    \det\left(\begin{bmatrix}
    1&-1&0\\
    1&0&-1
    \end{bmatrix}
    \begin{bmatrix}
    -3& 3\\
    -1& 4\\
     1& 4
    \end{bmatrix}\right) =
    \det\left(\begin{bmatrix}
     -2&-1 \\-4&-1
    \end{bmatrix}\right)=-2\neq 0
    \]
\end{itemize}
is full rank.
Note that our selection achieved bound on the number of driver nodes given in Corollary \ref{cor:min_drivers_b}, but not the minimum number of inputs (that are 6, applying Corollary \ref{cor:min_inputs}). This last achievement can easily be obtained replacing $D_1$ and $D_2$ with $D_1+D_7$ and $D_2+D_8$, and then removing $D_7$ and $D_8$ from $D$.

\bibliography{biblio_b.bib}
\bibliographystyle{apacite}
\end{document}